\documentclass[11pt]{amsart}
\usepackage[margin=1in]{geometry}

\usepackage{amssymb}
\usepackage{amsthm}
\usepackage{amsmath}
\usepackage{mathrsfs}
\usepackage{amsbsy}
\usepackage{bm}
\usepackage{hyperref}
\usepackage{tikz}
\usepackage{array}
\usepackage{enumerate}
\usepackage{bbm}
\usepackage{comment}
\usepackage{mathtools}
\usepackage{tabu}
\usepackage{makecell} 
\usepackage{colortbl}
\usepackage{xcolor}

\DeclareFontFamily{U}{mathx}{}
\DeclareFontShape{U}{mathx}{m}{n}{<-> mathx10}{}
\DeclareSymbolFont{mathx}{U}{mathx}{m}{n}
\DeclareMathAccent{\widecheck}{0}{mathx}{"71}

\definecolor{LightBlue}{rgb}{0.392,0.392,1} 
\definecolor{Green}{rgb}{0,0.922,0}
\definecolor{DarkGreen}{rgb}{0,0.5,0}
\definecolor{MildGreen}{rgb}{0,0.784,0}
\definecolor{NormalGreen}{rgb}{0,0.8,0}
\definecolor{LightGreen}{rgb}{0,0.922,0}
\definecolor{Magenta}{rgb}{1,0,0.6}
\definecolor{Yellow}{rgb}{0.95,0.95,0}
\definecolor{lavender}{rgb}{0.4,0,1}
\definecolor{peach}{rgb}{1,0.43,0.39} 
\definecolor{DarkPink}{rgb}{1,0,0.45} 

\usepackage[noabbrev,capitalise,nameinlink]{cleveref}

\hypersetup{colorlinks=true, citecolor=peach, linkcolor=peach,urlcolor=peach}

\usepackage{hhline}
\allowdisplaybreaks
\usepackage[noadjust]{cite}

\usepackage{caption}
\usepackage[noabbrev,capitalise,nameinlink]{cleveref}
\crefname{conjecture}{Conjecture}{Conjectures}

\newtheorem{theorem}{Theorem}[section]
\newtheorem{proposition}[theorem]{Proposition}
\newtheorem{corollary}[theorem]{Corollary}

\newtheorem{lemma}[theorem]{Lemma}

\theoremstyle{definition}

\newtheorem{remark}[theorem]{Remark}
\newtheorem{example}[theorem]{Example} 

\newcommand{\KK}{L}
\newcommand{\hh}{h}

\newcommand{\PPP}{\widecheck{P}}
\newcommand{\RRR}{\widecheck{R}}
\newcommand{\Z}{\mathcal{Z}}
\newcommand{\spann}{\mathrm{span}}
\renewcommand{\H}{\mathrm{H}}

\newcommand{\Cov}{\mathrm{Cov}}
\newcommand{\id}{\mathbbm{1}}
\newcommand{\PP}{\mathbb{P}}

\newcommand{\xx}{\xi}
\newcommand{\Q}{Q}

\newcommand{\EE}{\mathbb{E}}
\newcommand{\NN}{\mathcal{N}} 
\newcommand{\HH}{\mathcal{H}} 
\newcommand{\sub}{\mathsf{sub}}
\newcommand{\ww}{\mathsf{w}}
\renewcommand{\aa}{\mathsf{a}} 
\newcommand{\bb}{\mathsf{b}} 
\newcommand{\cc}{\mathsf{c}}
\renewcommand{\SS}{\mathfrak{S}}
\newcommand{\ZZ}{\mathbb{Z}}
\renewcommand{\AA}{\mathcal{A}} 
\newcommand{\zz}{\mathbf{z}}
 
\newcommand{\Tr}{\mathrm{Tr}} 
\newcommand{\ED}{\kappa}
\newcommand{\GG}{\nu} 
\newcommand{\II}{\mathrm{I}}

\newcommand\norm[1]{\left\lVert#1\right\rVert}

\newcommand{\dfn}[1]{\textcolor{DarkPink}{\emph{#1}}}

\definecolor{pumpkin}{RGB}{255,117,24}

\begin{document}

\title[]{Random Subwords and Billiard Walks in Affine Weyl Groups} 
\subjclass[2010]{}

\author[]{Colin Defant}
\address[]{Department of Mathematics, Harvard University, Cambridge, MA 02138, USA}
\email{colindefant@gmail.com} 

\author[]{Pakawut Jiradilok}
\address[]{Department of Mathematics, Massachusetts Institute of Technology, Cambridge, MA 02139}
\email{pakawut@mit.edu} 

\author[]{Elchanan Mossel} 
\address[]{Department of Mathematics, Massachusetts Institute of Technology, Cambridge, MA 02139}
\email{elmos@mit.edu}

\begin{abstract}
Let $W$ be an irreducible affine Weyl group, and let $\mathsf{b}$ be a finite word over the alphabet of simple reflections of $W$. Fix a probability $p\in(0,1)$. For each integer $K\geq 0$, let $\mathsf{sub}_p(\mathsf{b}^K)$ be the random subword of $\mathsf{b}^K$ obtained by deleting each letter independently with probability $1-p$. Let $v_p(\mathsf{b}^K)$ be the element of $W$ represented by $\mathsf{sub}_p(\mathsf{b}^K)$. One can view $v_p(\mathsf{b}^K)$ geometrically as a random alcove; in many cases, this alcove can be seen as the location after a certain amount of time of a random billiard trajectory that, upon hitting a hyperplane in the Coxeter arrangement of $W$, reflects off of the hyperplane with probability $1-p$. We show that the asymptotic distribution of $v_p(\mathsf{b}^K)$ is a central spherical multivariate normal distribution with some variance $\sigma_{\mathsf{b}}^2$ depending on $\mathsf{b}$ and $p$. We provide a formula to compute $\sigma_{\mathsf{b}}^2$ that is remarkably simple when $\mathsf{b}$ contains only one occurrence of the simple reflection that is not in the associated finite Weyl group.
As a corollary, we provide an asymptotic formula for $\mathbb{E}[\ell(v_p(\mathsf{b}^K))]$, the expected Coxeter length of $v_p(\mathsf{b}^K)$. For example, when $W=\widetilde A_{r}$ and $\mathsf{b}$ contains each simple reflection exactly once, we find that 
\[\lim_{K\to\infty}\frac{1}{\sqrt{K}}\mathbb{E}[\ell(v_p(\mathsf{b}^K))]=\sqrt{\frac{2}{\pi}r(r+1)\frac{p}{1-p}}.\] 
\end{abstract} 

\maketitle

\section{Introduction}\label{sec:intro} 

\begin{figure}[ht]
  \begin{center}{\includegraphics[width=\linewidth]{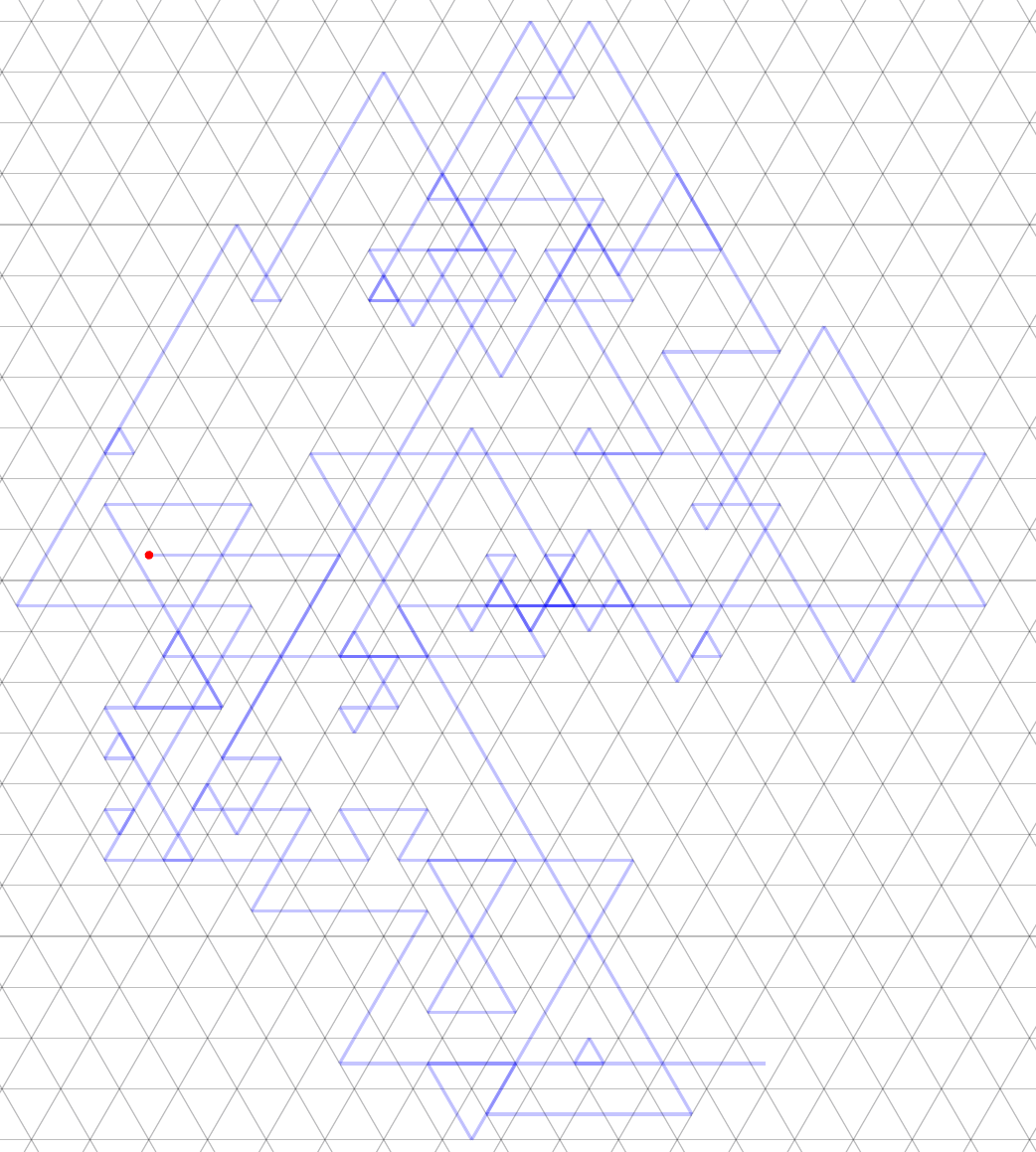}}
  \end{center}
\caption{The trajectory of a random billiard walk corresponding to a random subword $\sub_{4/5}((s_2s_1s_0)^{200})$ in type $\widetilde A_2$. Line segments traversed more frequently are drawn darker. The red dot indicates the starting point $\zz_0$.}\label{fig:A2_600}
\end{figure}

\subsection{Random Subwords} 
Our story begins with a natural question that merges Coxeter theory with probability theory. Let $(W,S)$ be a Coxeter system, and let $\ww$ be a word over the alphabet $S$. Fix a probability $p\in(0,1)$. Let $\sub_p(\ww)$ be the random subword of $\ww$ obtained by deleting each letter of $\ww$ independently with probability $1-p$. What can be said about the distribution of the element $v_p(\ww)$ of $W$ represented by $\sub_p(\ww)$? 

This question is very broad, so we should narrow our scope in order to have any hope of saying something interesting. For example, suppose $W$ is the symmetric group $\SS_n$ and $\ww_{\mathrm{stair}}$ is the \dfn{staircase reduced word} $(s_{n-1})(s_{n-2}s_{n-1})\cdots(s_1s_2\cdots s_{n-1})$, where $s_i$ is the transposition $(i\,\,i+1)$. In this very special case, one can view $\sub_p(\ww_{\mathrm{stair}})$ as a \emph{random pipe dream} (see \cite{DefantPipes,MPPY}). Since $\ww_{\mathrm{stair}}$ is reduced, it is natural to ask how ``far'' $\sub_p(\ww_{\mathrm{stair}})$ is from being reduced; one way to measure this is to compute the expected Coxeter length (i.e., the expected number of inversions) of the random permutation $v_p(\ww_{\mathrm{stair}})$ represented by $\sub_p(\ww_{\mathrm{stair}})$. In \cite{MPPY}, Morales, Panova, Petrov, and Yeliussizov obtained asymptotic bounds (as $n\to\infty$) for this expected Coxeter length; settling one of their conjectures, Defant \cite{DefantPipes} obtained an exact asymptotic formula. Defant's method also handles a variety of other choices of the starting word $\ww$ besides the staircase reduced word. 

In the present article, we focus on the case where $W$ is an irreducible affine Weyl group. Because such a group is infinite, the techniques we employ and the behavior we observe will be much different from in the aforementioned setting of the symmetric group. 
We will focus primarily on the setting in which $\ww=\bb^K$, where $\bb$ is a fixed finite word over $S$ and $K$ is a large positive integer that we imagine is growing. 

\subsection{Random Billiard Walks}\label{subsec:billiards} 

Let $W$ be an irreducible affine Weyl group. We can view $W$ as the group of affine transformations of a Euclidean space $V$ generated by the reflections through the hyperplanes in a certain affine hyperplane arrangement $\HH_W$. We make the convention that this defining action of $W$ on $V$ is a \emph{right} action. The closures of the connected components of $V\setminus\bigcup\HH_W$ are pairwise-congruent simplices called \dfn{alcoves}. There is a particular alcove $\AA$ called the \emph{fundamental alcove}. The (right) action of $W$ on $V$ induces a free and transitive action on the set of alcoves, so the map $u\mapsto \AA u$ is a bijection from $W$ to the set of alcoves. The alcoves adjacent to $\AA u$ are those of the form $\AA su$ for $s\in S$. We denote by $u^\bullet$ the centroid of the alcove $\AA u$. 

Fix a point $\zz_0$ in the interior of $\AA$, and shine a beam of light in the direction of some vector $\eta\in V$; the beam of light travels along a (straight) ray in $V$ that starts at $\zz_0$. Let us assume that the beam of light never passes through the intersection of two or more distinct hyperplanes in $\HH_W$. Let $\AA y_0,\AA y_1,\AA y_2,\ldots$ be the sequence of alcoves through which the beam of light passes. In particular, $y_0$ is the identity element of $W$, which we denote by $\id$. We obtain an infinite word\footnote{Our convention is that infinite words extend infinitely to the left.} $\ww_{\zz_0}(\eta)=\cdots s_{i_2}s_{i_1}$, where $s_{i_j}\in S$ is the unique simple reflection such that $y_{j}=s_{i_j}y_{j-1}$. 

\begin{figure}[ht]
  \begin{center}{\includegraphics[width=\linewidth]{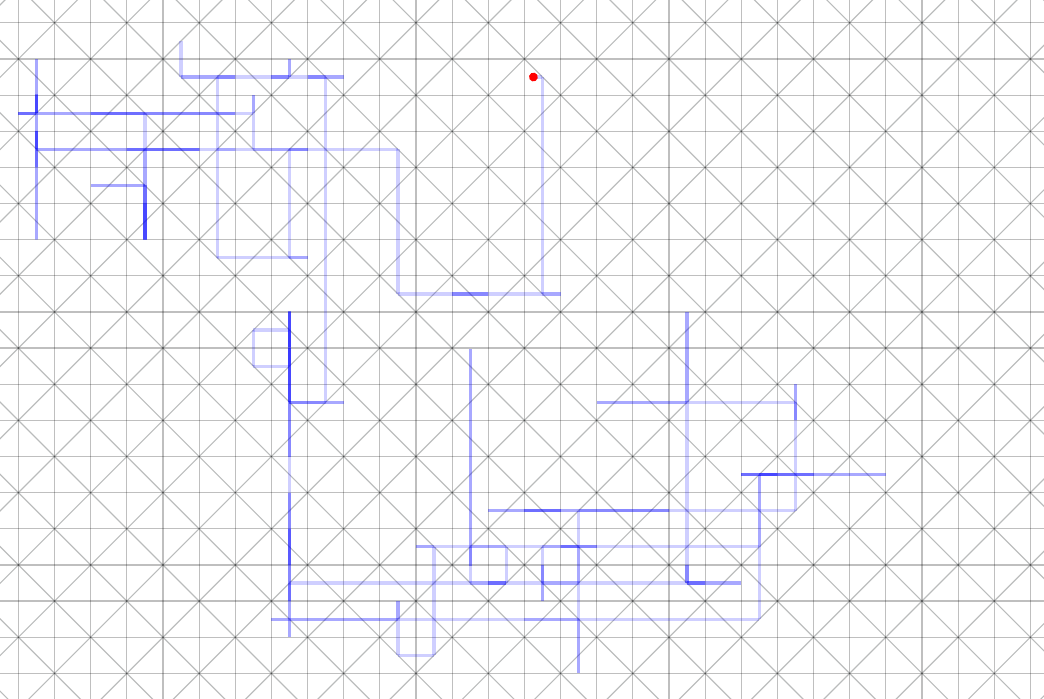}}
  \end{center}
\caption{The trajectory of a random billiard walk corresponding to a random subword $\sub_{4/5}((s_1s_2s_1s_0)^{150})$ in type $\widetilde C_2$. Line segments traversed more frequently are drawn darker. The red dot indicates the starting point $\zz_0$.}\label{fig:C2_600}
\end{figure} 

We now introduce a random process that we call a \dfn{random billiard walk}. As before, start by shining a beam of light from $\zz_0$ in the direction of $\eta$. But now, whenever the beam of light hits a hyperplane $\H\in\HH_W$, it passes through $\H$ with probability $p$ and reflects off of $\H$ with probability $1-p$. (See \cref{fig:A2_600,fig:C2_600}.) We are interested in the distribution of the location of the beam of light after a long period of time has passed. 

A crucial observation is that we can model the location of the beam of light algebraically using random subwords of suffixes of the word $\ww_{\zz_0}(\eta)=\cdots s_{i_2}s_{i_1}$. When the beam hits a hyperplane in $\HH_W$ for the first time, it passes from $\AA$ to $\AA s_{i_1}$ with probability $p$, and it reflects---thereby staying in $\AA$---with probability $1-p$. Therefore, the location of the beam of light after it has hit one hyperplane in $\HH_W$ is the alcove $\AA v_p(s_{i_1})$. When the beam hits a hyperplane in $\HH_W$ for the second time, it will either pass from $\AA v_p(s_{i_1})$ to $\AA s_{i_2}v_p(s_{i_1})$ (with probability $p$) or reflect and stay in the alcove $\AA v_p(s_{i_1})$ (with probability $1-p$). Therefore, the location of the beam of light after it has hit two hyperplanes in $\HH_W$ is the alcove $\AA v_p(s_{i_2}s_{i_1})$. In general, the alcove containing the beam of light immediately after the beam has hit a hyperplane in $\HH_W$ for the $M$-th time is $\AA v_p(s_{i_{M}}\cdots s_{i_2}s_{i_1})$. 

\begin{example}
Let $W=\widetilde A_2$. Then $V$ is a $2$-dimensional Euclidean space, and the alcoves are the unit triangles in the standard equilateral triangular grid, as depicted in \cref{fig:A2_small}. In that figure, the line segment separating an alcove $\AA u$ from the adjacent alcove $\AA s_iu$ is labeled with the index $i$. Let $\zz_0$ be the point indicated by the red dot, and let $\eta$ be a vector pointing to the right. The resulting infinite word $\ww_{\zz_0}(\eta)$ is $\cdots s_2s_1s_0s_2s_1s_0=(s_2s_1s_0)^\infty$. In blue is the trajectory of the random billiard walk corresponding to the subword $s_2s_1s_0s_2\,\underline{\,\,\,\,\,}\,\underline{\,\,\,\,\,}\,s_2s_1\,\underline{\,\,\,\,\,}\,s_2\,\underline{\,\,\,\,\,}\,s_0s_2s_1s_0$ of $(s_2s_1s_0)^5$ (underscores represent deleted letters).
\end{example}

\begin{figure}[ht]
  \begin{center}{\includegraphics[width=0.6\linewidth]{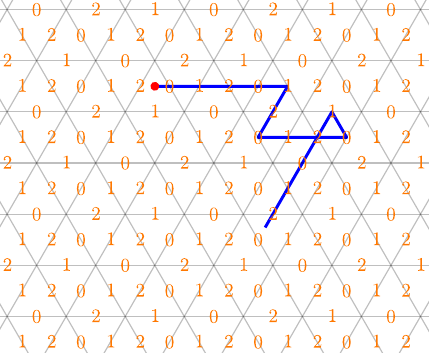}}
  \end{center}
\caption{The trajectory of the billiard walk corresponding to the subword $s_2s_1s_0s_2\,\underline{\,\,\,\,\,}\,\underline{\,\,\,\,\,}\,s_2s_1\,\underline{\,\,\,\,\,}\,s_2\,\underline{\,\,\,\,\,}\,s_0s_2s_1s_0$ of $(s_2s_1s_0)^5$ in type $\widetilde A_2$. The red dot indicates the starting point $\zz_0$.}\label{fig:A2_small} 
\end{figure}

As mentioned above, we will primarily focus on random subwords of the form $\sub_p(\bb^K)$, where $\bb$ is a fixed finite word. Such words arise from random billiard walks with ``rational'' initial directions. More precisely, the word $\ww_{\zz_0}(\eta)$ is periodic if and only if $\eta$ is a positive scalar multiple of a \emph{coroot vector}. However, there are other natural periodic infinite words that do not come from beams of light that move in straight lines. For example, suppose $\bb$ is a word that uses each simple reflection exactly once (so $\bb$ is a reduced word for a \emph{standard Coxeter element} of $W$). We can ask whether there exist a point $\zz_0$ and a vector $\eta$ such that the word $\ww_{\zz_0}(\eta)$ is equal to the word $\bb^\infty$ obtained by concatenating $\bb$ with itself infinitely many times. For some choices of $\bb$, the answer is affirmative, whereas for other choices, the answer is negative. However, it \emph{is} known \cite{Speyer} that every word of the form $\bb^K$ is \emph{reduced}. This essentially means that we can obtain $\bb^\infty$ by following a curve that travels through an infinite sequence of alcoves, never passing through any hyperplane in $\HH_W$ more than once. Informally, this means that we can repeat the same story about the beam of light from before, except now we follow a beam of ``pseudolight,'' which is a fictitious variant of light that moves along a squiggly path. 

\begin{remark}
It is worth mentioning another slightly different way of thinking about the distribution of $v_p(\bb^K)$ in terms of light beams. Instead of imagining that the entire beam of light either passes through or reflects off of a hyperplane, imagine instead that a proportion $p$ of the light passes through while a proportion $1-p$ of the light reflects. (This perspective is inspired by the behavior of light in the real world.) Over time, the light will distribute throughout the space $V$. By identifying elements of $W$ with alcoves, we can see the distribution of $v_p(\bb^K)$ as a description of how this light is distributed after a certain amount of time has passed. 
\end{remark} 

\begin{remark}
The element $v_p(\bb^K)$ has the same distribution as $v^{(K)}\cdots v^{(1)}$, where $v^{(1)},\ldots,v^{(K)}$ are independent random elements of $W$, each with the same distribution as $v_p(\bb)$. Thus, one can formulate our model in the language of random walks on groups \cite{Woess}. However, let us note that our random walk is not symmetric in general. More precisely, there can exist $y\in W$ such that $\PP(v_p(\bb)=y)\neq\PP(v_p(\bb)=y^{-1})$. This means that we cannot immediately appeal to classical results that are stated specifically for symmetric random walks. In addition, several of our main results (\cref{thm:covariance_complicated}, \ref{thm:covariance_simple}, and \ref{thm:table}) are much more precise than results typically proven about random walks on groups. 
\end{remark}

\subsection{Main Results} 

Let $\Phi$ be a finite irreducible crystallographic root system spanning an $r$-dimensional Euclidean space $V$, and let $\langle\cdot,\cdot\rangle$ be the inner product on $V$. We write $\Phi^+$ and $\Phi^-$ for the set of positive roots and the set of negative roots, respectively, so that $\Phi^-=-\Phi^+$ and $\Phi=\Phi^+\sqcup\Phi^-$. Let $W$ and $\overline W$ be the associated affine and finite Weyl groups, respectively. Let $S$ be the set of simple reflections of $W$, and let $s_0\in S$ be the simple reflection of $W$ that is not in $\overline W$. Note that $r=|S|-1$. Let \[\bb=s_{i_m}\cdots s_{i_1}\] be a finite word over $S$, and define $i_j$ for all integers $j>m$ by setting $i_{j}=i_{j-m}$. Fix a probability $p\in(0,1)$. For each integer $K\geq 0$, let $v_p(\bb^K)$ be the random element of $W$ represented by $\sub_p(\bb^K)$. 

If there is a simple reflection $s^*\in S$ that does not appear in the word $\bb$, then $v_p(\bb^K)$ is confined to the group generated by $S\setminus\{s^*\}$, which is necessarily finite. This setting is not so interesting for our purposes. Therefore, we will assume for simplicity throughout the rest of the article that $\bb$ contains each simple reflection in $S$ at least once. 

In what follows, we fix an orthonormal basis $\mathcal E=\{\varepsilon_1,\ldots,\varepsilon_r\}$ of $V$ and compute covariance matrices with respect to that basis. Let $\II_k$ be the $k\times k$ identity matrix. 

\begin{theorem}\label{thm:normal}
Fix a finite word $\bb$ over $S$ that uses each simple reflection at least once. As $K\to\infty$, the distribution of $v_p(\bb^K)^\bullet/\sqrt{K}$ converges in distribution to a multivariate normal distribution $\NN(0,\sigma^2_\bb \II_r)$, where $\sigma_\bb$ is a positive real number depending only on $\bb$ and $p$.  
\end{theorem} 

Let $\mathscr{Y}_1,\mathscr{Y}_2,\ldots$ be independent random variables with values in $W$, where $\mathscr{Y_j}$ has the same distribution as $v_p(s_{i_j})$. Note that $\mathscr{Y}_{Km}\cdots\mathscr{Y}_1$ has the same distribution as $v_p(\bb^K)$. As mentioned above, we can view $\AA \mathscr{Y}_M\cdots\mathscr{Y}_1$ as the alcove containing the random billiard walk associated with $\bb$ (which might be made of ``pseudolight'') after the beam has hit a hyperplane in $\HH_W$ for the $M$-th time. Let us say the random billiard walk associated with $\bb$ is \dfn{recurrent} if with probability $1$, there are infinitely many positive integers $M$ such that $\mathscr{Y}_M\cdots\mathscr{Y}_1=\id$; otherwise, say it is \dfn{transient}. As a corollary of \cref{thm:normal}, we will characterize exactly when the random billiard walk is recurrent.

\begin{corollary}\label{cor:recurrent}
Let $\bb$ be a word that uses each simple reflection in $S$ at least once. The random billiard walk associated to the word $\bb$ is recurrent if and only if $r\leq 2$.    
\end{corollary}

The next corollary concerns the asymptotic distribution of the Coxeter length $\ell(v_p(\bb^K))$, which can be derived from the asymptotic distribution of $v_p(\bb^K)^\bullet$ described in \cref{thm:normal}. For $\gamma\in V$, let $\norm{\gamma}=\sqrt{\langle\gamma,\gamma\rangle}$.

\begin{corollary}\label{cor:length}
As $K\to\infty$, we have \[\mathbb E[\ell(v_p(\bb^K))]=\sqrt{\frac{2}{\pi}}\,\sigma_\bb\sum_{\beta\in\Phi^+}\norm{\beta}\sqrt{K}+o(\sqrt{K}).\] 
\end{corollary}

In order to truly understand the asymptotic behavior of $v_p(\bb^K)$, we wish to compute $\sigma^2_\bb$ exactly. We will do so by leveraging the high amount of symmetry coming from the affine Weyl group $W$.  

Let $\Q^\vee$ denote the coroot lattice of $W$, and let $\theta^\vee$ be the coroot associated with the highest root $\theta$. There is a standard semidirect product decomposition $W=\overline{W}\ltimes \Lambda$, where $\Lambda$ is a subgroup of $W$ that is naturally isomorphic to $\Q^\vee$. This yields a natural quotient map $W\to\overline{W}$, which we denote by $u\mapsto\overline{u}$. For $w\in\overline W$, let $P_w$ be the linear transformation of $V$ corresponding to the right action of $w$; that is, $P_w\gamma=\gamma w$ for all $\gamma\in V$. For $s\in S$, let $R_s=(1-p)\II_V+p{P}_{\overline s}$, where $\II_V$ is the identity map on $V$. More generally, given a finite word $\aa=s_{j_k}\cdots s_{j_1}$ over $S$, let $R_{\aa}=R_{s_{j_1}}\cdots R_{s_{j_k}}$ (note the reversed order). For each $J=\{j_1<\cdots<j_M\}\subseteq[m]$, we define the element $\overline s_J=\overline s_{i_{j_M}}\cdots\overline s_{i_{j_1}}\in \overline W$. 

Let $\Z=\{j\in[m]:i_j=0\}$. For $J\subseteq[m]$, let $\GG_J=\sum_{k\in J\cap\Z}{P}_{\overline s_{J\cap[k-1]}}\theta^\vee$. Using the assumption that $\bb$ contains every simple reflection, one can prove (see \cref{lem:invertible}) that $\II_V-R_\bb$ is invertible. Let \[\chi_\bb=\sum_{J\subseteq[m]}p^{|J|}(1-p)^{m-|J|}\norm{\GG_J}^2\quad\text{and}\quad \ED_\bb=|\overline W|(\II_V-R_\bb)^{-1}\sum_{J\subseteq[m]}p^{|J|}(1-p)^{m-|J|}\GG_J.\] 
Both of these quantities can be computed readily from the word $\bb$. Hence, the next theorem provides an computationally efficient formula for $\sigma_\bb^2$.

\begin{theorem}\label{thm:covariance_complicated}
Let $\bb=s_{i_m}\cdots s_{i_1}$ be a word that contains each simple reflection in $S$ at least once. We have 
\[\sigma_\bb^2=\frac{1}{r}\left(\frac{2}{|\overline W|}\sum_{J\subseteq[m]}p^{|J|}(1-p)^{m-|J|}\langle\GG_J,{P}_{\overline s_J}\ED_\bb\rangle+\chi_\bb\right).\]
\end{theorem} 

If we assume that the word $\bb$ contains exactly one occurrence of the simple reflection $s_0$, then the formula in \cref{thm:covariance_complicated} simplifies substantially. 

\begin{theorem}\label{thm:covariance_simple}
Let $\bb=s_{i_m}\cdots s_{i_1}$ be a word over $S$ that uses each simple reflection at least once and uses $s_0$ exactly once. Let $\aa$ be the unique cyclic rotation of $\bb$ whose rightmost letter is $s_0$. Then 
\[\sigma_\bb^2=\frac{1}{r}\lim_{q\to p}\frac{q}{2q-1}\langle\theta^\vee,(2q(\II_V-R_\aa)^{-1}-\II_V)\theta^\vee\rangle.\]
\end{theorem}

\begin{remark}
The limit in \cref{thm:covariance_simple} is only needed so the formula makes sense when $p=1/2$. 
\end{remark}

A \dfn{Coxeter word} is a word over $S$ that uses each simple reflection exactly once. To illustrate the utility of \cref{thm:covariance_simple}, we will apply it to derive explicit formulas for $\sigma_\bb^2$ when $\bb$ is a Coxeter word. 

\begin{theorem}\label{thm:table}
If $\bb$ is a Coxeter word, then $\sigma_\bb^2$ and $\lim_{K\to\infty}\EE[\ell(v_p(\bb^K))]/\sqrt{K}$ are as in~\cref{table:sigma_b}. 
\end{theorem} 

\begin{table}[h!]
\centering
\setlength{\arrayrulewidth}{0.5mm} 
\setlength{\tabcolsep}{12pt} 
\renewcommand{\arraystretch}{2.3} 
\begin{tabular}{|c|c|c|} 
\hline
\rowcolor{orange!15} 
$W$ & $\sigma_\bb^2$ & $\displaystyle{\lim\limits_{K\to\infty}\frac{\EE[\ell(v_p(\mathsf{b}^K))]}{\sqrt{K}}}$ \\[1.5ex]
\hline
$ \widetilde A_r $ & $ \displaystyle{\frac{2}{r(r+1)}\frac{p}{1-p}} $ & $ \displaystyle{\sqrt{\frac{2}{\pi}r(r+1)\frac{p}{1-p}}} $ \\[1.5ex]
\hline
$ \widetilde B_r $ & $ \displaystyle{\frac{1}{2r(r-1)}\frac{p}{1-p}} $ & $ \displaystyle{\left((r-1)\sqrt{2}+1\right)\sqrt{\frac{1}{\pi}\frac{r}{r-1}\frac{p}{1-p}}} $ \\[1ex]
\hline
$ \widetilde C_r $ & $ \displaystyle{\frac{1}{2r^2}\frac{p}{1-p}} $ & $ \displaystyle{\left((r-1)\sqrt{2}+2\right)\sqrt{\frac{1}{\pi}\frac{p}{1-p}}} $ \\[1ex]
\hline
$ \widetilde D_r $ & $ \displaystyle{\frac{1}{2r(r-2)}\frac{p}{1-p}} $ & $ \displaystyle{(r-1)\sqrt{\frac{2}{\pi}\frac{r}{r-2}\frac{p}{1-p}}} $ \\[1.5ex]
\hline
$ \widetilde E_6 $ & $ \displaystyle{\frac{1}{72}}\frac{p}{1-p} $ & $ \displaystyle{6\,\sqrt{\frac{2}{\pi}\frac{p}{1-p}}} $ \\[1ex]
\hline
$ \widetilde E_7 $ & $ \displaystyle{\frac{1}{168}}\frac{p}{1-p} $ & $ \displaystyle{3\,\sqrt{\frac{21}{2\pi}\frac{p}{1-p}}} $ \\[1ex]
\hline
$ \widetilde E_8 $ & $ \displaystyle{\frac{1}{480}}\frac{p}{1-p} $ & $ \displaystyle{2\,\sqrt{\frac{30}{\pi}\frac{p}{1-p}}} $ \\[1ex]
\hline
$ \widetilde F_4 $ & $ \displaystyle{\frac{1}{48}}\frac{p}{1-p} $ & $ \displaystyle{(\sqrt{2}+1)\sqrt{\frac{6}{\pi}\frac{p}{1-p}}} $  \\[1ex]
\hline
$ \widetilde G_2 $ & $ \displaystyle{\frac{1}{24}}\frac{p}{1-p} $ & $ \displaystyle{(\sqrt{3}+1)\sqrt{\frac{3}{2\pi}\frac{p}{1-p}}} $ \\[1ex]
\hline
\end{tabular}
\caption{Quantities describing the asymptotic behavior of $v_p(\bb^K)$ for the irreducible affine Weyl groups, where $\bb$ is a Coxeter word.}\label{table:sigma_b} 
\end{table}

\begin{remark}
If $\bb$ is a Coxeter word, then it follows from \cref{thm:table} that $\sigma_\bb^2$ is of the form $\varkappa_W\frac{p}{1-p}$, where $\varkappa_W$ is a real number that only depends on $W$ (and not $\bb$ or $p$). The fact that $\sigma_\bb^2$ does not depend on the specific order of the letters in $\bb$ is unsurprising when $W$ is not of type $\widetilde A_r$ for $r\geq 2$ since, in this case, it is a straightforward consequence of the fact that the Coxeter graph of $W$ is a tree (see \cref{rem:flip}). It is also unsurprising when $W=\widetilde A_2$ since the Coxeter graph of $\widetilde A_2$ is a triangle, which is highly symmetric. However, it is somewhat surprising when $W$ is of type $\widetilde A_r$ for $r\geq 3$. 

On the other hand, if $\bb$ uses some simple reflections more than once, then $\sigma_\bb^2$ might depend on the specific order of the letters in $\bb$, and $\sigma_\bb^2\frac{1-p}{p}$ might depend on $p$. For example, when $W=\widetilde A_2$, we can use \cref{thm:covariance_simple} to find that 
\[\sigma_{s_1s_2s_1s_0}^2=\frac{2}{5}\frac{p}{1-p}\neq\frac{2p(1-p)}{5-10p+6p^2}=\sigma_{s_2s_1s_1s_0}^2.\]
\end{remark}

\subsection{Related Previous Work}

Given a Coxeter system $(W,S)$ and a word $\mathsf{v}$ over $S$, we can of course multiply the letters in $\mathsf{v}$ using the usual group product in order to obtain an element of $W$. However, there is another natural product known as the \emph{Demazure product}. One could also study the distribution of the random element of $W$ obtained by taking the Demazure product of the random subword $\sub_p(\ww)$. When $W=\SS_n$ and $\ww=\ww_{\mathrm{stair}}$ is the staircase reduced word, this problem was considered in \cite{MPPY}. When $W$ is an affine Weyl group and $\ww=\bb^K$ for some finite word $\bb$, this problem was considered in \cite{DefantStoned} (in the regime where $K \to \infty$). It is notable that both of these settings relate to versions of the totally asymmetric simple exclusion process (TASEP). 

This paper can be seen as a continuation of a recent program devoted to understanding \emph{combinatorial billiard systems}, which are, roughly speaking, billiard systems that are rigid and discretized \cite{DefantRefractions, DefantStoned, DefantJiradilok, BDHKL, Zhu}. Such systems concern a beam of light that travels in a particular direction and can reflect (or \emph{refract} \cite{DefantRefractions}) when it hits a hyperplane in the Coxeter arrangement of a Coxeter group.  

\subsection{Outline}
\cref{sec:preliminaries} provides background information from Coxeter theory and probability theory. In \cref{sec:normality}, we prove \cref{thm:normal,cor:recurrent,cor:length}. We prove \cref{thm:covariance_complicated,thm:covariance_simple} in \cref{sec:computing}, and we prove \cref{thm:table} in \cref{sec:explicit}. Finally, we end with ideas for future work in \cref{sec:conclusion}.

\section{Preliminaries}\label{sec:preliminaries} 
\subsection{Finite and Affine Weyl Groups} 
As before, let $\Phi$ be a finite irreducible crystallographic root system spanning an $r$-dimensional Euclidean space $V$. Let $\alpha_1,\ldots,\alpha_r$ be the simple roots. Associated to each root $\beta\in\Phi$ is the \dfn{coroot} $\beta^\vee=\frac{2}{\norm{\beta}^2}\beta$. The \dfn{coroot lattice} is ${\Q^\vee=\spann_\ZZ\{\beta^\vee:\beta\in\Phi\}}$. Let $\theta\in\Phi^+$ denote the highest root, and let $\theta^\vee$ be its associated coroot.

For $\beta\in\Phi$ and $k\in\ZZ$, we consider the affine hyperplane 
\[\H_\beta^k=\{\gamma\in V:\langle\beta,\gamma\rangle=k\}.\] Let $s_{\beta,k}$ denote the orthogonal reflection through $\H_\beta^k$. Let $s_0=s_{\theta,1}$, and for $1\leq i\leq r$, let $s_i=s_{\alpha_i,0}$. The \dfn{affine Weyl group} of $\Phi$ is the group $W$ generated by the set $S=\{s_0,s_1,\ldots,s_r\}$ (equivalently, by all reflections of the form $s_{\beta,k}$ for $\beta\in\Phi$ and $k\in\ZZ$), while the \dfn{Weyl group} of $\Phi$ is the subgroup $\overline W$ of $W$ generated by $S\setminus\{s_0\}$ (equivalently, by all reflections of the form $s_{\beta,0}$ for $\beta\in\Phi$). We find it convenient to view the defining actions of $W$ and $\overline W$ on $V$ as right actions. The elements of $S$ are called the \dfn{simple reflections} of $W$. We denote by $\id$ the identity element of $W$ (and of $\overline W$). 
The \dfn{Coxeter graph} of $W$ is the graph with vertex set $S$ in which $s_i$ and $s_{i'}$ are adjacent if and only if $s_i$ and $s_{i'}$ do not commute; if $s_{i}$ and $s_{i'}$ are adjacent and the order of $s_is_{i'}$ is not $3$, then the edge between them is labeled by the order of $s_is_{i'}$. 

The hyperplane arrangement $\HH_W=\{\H_\beta^k\colon\beta\in\Phi,\,k\in\ZZ\}$ is called the \dfn{Coxeter arrangement} of $W$. The closures of the connected components of $V\setminus\bigcup\HH_W$ are pairwise-congruent simplices called \dfn{alcoves}. The \dfn{fundamental alcove}, denoted by $\AA$, is the alcove bounded by the hyperplanes $\H_{\alpha_i,0}$ for $1\leq i\leq r$ and the hyperplane $\H_{\theta,1}$. The defining (right) action of $W$ on $V$ induces a free and transitive (right) action of $W$ on the set of alcoves. In other words, the map $u\mapsto\AA u$ is a bijection from $W$ to the set of alcoves. Two distinct alcoves are \dfn{adjacent} if they share a common facet. For each $u\in W$, the alcoves adjacent to $\AA u$ are those of the form $\AA s_iu$ for $0\leq i\leq r$. 

Because $\Q^\vee$ is a lattice in $V$, we can form the quotient space $V/\Q^\vee$, which is an $r$-dimensional torus. For each $\eta\in \Q^\vee$, there is a unique element $\tau_\eta\in W$ that acts on $V$ via translation by $-\eta$; that is, $\gamma\tau_\eta=\gamma-\eta$ for all $\gamma\in V$. Let $\Lambda=\{\tau_\eta:\eta\in\Q^\vee\}$. The map $\eta\mapsto\tau_\eta$ is a group isomorphism from $\Q^\vee$ to $\Lambda$. The group $W$ decomposes as the semidirect product $W=\overline W\ltimes \Lambda$, so we obtain a quotient map $W\to\overline W$ denoted by $u\mapsto\overline u$. For each $u\in W$, there is a unique $\xx(u)\in \Q^\vee$ such that $u=\overline u\tau_{-\xx(u)}$. The alcove $\AA u$ is obtained by translating $\AA\overline u$ by $\xx(u)$; moreover, $\xx(u)$ is the unique coroot vector contained in $\AA u$. One case of particular importance for us is when $u$ is the simple reflection $s_0$. We have $\xx(s_0)=\theta^\vee$, so $s_0=\overline s_0\tau_{-\theta^\vee}$.  

A \dfn{reduced word} for an element $u\in W$ is a minimum-length word over $S$ that represents $u$. The length of a reduced word for $u$ is called the \dfn{Coxeter length} of $u$ and is denoted $\ell(u)$. Equivalently, $\ell(u)$ is the number of hyperplanes in $\mathcal H_W$ that separate $u^\bullet$ (the centroid of $\AA u$) from $\id^\bullet$ (the centroid of $\AA$). 

\subsection{Multivariate Distributions} 
Let us fix an orthonormal basis $\mathcal E=\{\varepsilon_1,\ldots,\varepsilon_r\}$
of $V$. We will frequently view elements of $V$ as column vectors with respect to this basis. More generally, we will write matrices with respect to this basis. We denote the transpose of a matrix $A$ by $A^\top$. We denote a probability by $\PP$ and denote an expected value by $\EE$. Suppose $X=(X_1,\ldots,X_r)^\top\in V$ is a random vector; so $X_i=\langle X,\varepsilon_i\rangle$ is a real-valued random variable. The \dfn{covariance matrix} of $X$ is the $r\times r$ matrix 
\[\Cov(X)=\EE[(X-\EE[X])(X-\EE[X])^\top].\] 

We denote by $\NN(0,\boldsymbol{\Sigma})$ the multivariate normal distribution on $V$ with mean $0$ and covariance matrix $\boldsymbol{\Sigma}$. When this distribution is nondegenerate (i.e., $\boldsymbol{\Sigma}$ is positive definite), its density function $f_{\boldsymbol{\Sigma}}$ is given by 
\[f_{\boldsymbol{\Sigma}}(\gamma)=\frac{\exp\left(-\frac{1}{2}\langle\gamma,\boldsymbol{\Sigma}^{-1}\gamma\rangle\right)}{\sqrt{(2\pi)^r\det(\boldsymbol{\Sigma})}}.\]

\subsection{Other Notation}
For integers $a,b$, we let $[a,b]=\{i\in\ZZ:a\leq i\leq b\}$. Given a statement $\mathrm{P}$, we let  
\[\delta_{\mathrm{P}}=\begin{cases}
    0 & \text{if P is false}; \\
    1 & \text{if P is true}.
\end{cases}\] 

\section{Asymptotic Normality}\label{sec:normality}

Fix a length-$m$ word $\bb=s_{i_{m}}\cdots s_{i_1}$ over $S$. For each subset ${J=\{j_1<\cdots<j_M\}\subseteq[m]}$, let \[s_J=s_{i_{j_M}}\cdots s_{i_{j_1}}\quad\text{and}\quad\overline s_J=\overline s_{i_{j_M}}\cdots\overline s_{i_{j_1}}.\] Let $v^{(1)},v^{(2)},\ldots$ be independent copies of $v_p(\bb)$. For each integer $K\geq 0$, let $u_K=v^{(K)}\cdots v^{(1)}$. Note that $u_K$ has the same distribution as the element $v_p(\bb^K)$ of $W$ represented by $\sub_p(\bb^K)$. Recall that we write $u_K^\bullet$ for the centroid of the alcove $\AA u_K$. 

A crucial tool for us is the fact that the sequence $(\overline u_K)_{K\geq 0}$ is a Markov chain $\mathscr M$ with (finite) state space $\overline W$. Its transition probabilities are such that 
\begin{equation}\label{eq:P-u-K-1-u-K-transition}
\PP[ \overline u_{K+1} = w' \mid \overline u_K = w ] = \sum_{\substack{J\subseteq[0,m-1]\\ w'=\overline s_Jw}}p^{|J|}(1-p)^{m-|J|},
\end{equation}
for all $K\geq 0$. This Markov chain is irreducible and aperiodic since $\{\overline s_J:J\subseteq[m]\}$ generates $\overline W$ and contains $\id$. For all $w,w',x\in\overline W$, we have $\PP[ \overline u_{K+1} = w' \mid \overline u_K = w ]=\PP[ \overline u_{K+1} = w'x \mid \overline u_K = wx ]$. This implies that the stationary distribution of $\mathscr{M}$ is invariant under the right action of $\overline W$ on itself, so it must be the uniform distribution on $\overline W$. For $w\in\overline W$, let $\mathcal L_w=\{L\geq 0:\overline u_L=w\}$. We will assume that each set $\mathcal L_w$ is infinite (this happens with probability~$1$). 

Let $w,w'\in\overline W$, and let $\boldsymbol{t}=\min\mathcal{L}_w$. Let $T_w^{w'}$ denote the smallest positive integer such that $\overline u_{\boldsymbol{t}+T_w^{w'}}=w'$, and let \[D_w^{w'}=\xx\left(u_{\boldsymbol{t}+T_w^{w'}}\right)-\xx\left(u_{\boldsymbol{t}}\right).\] Note that $T_w^{w'}$ has the same distribution as the first-passage time from $w$ to $w'$ in $\mathscr{M}$; this distribution depends only on $w$ and $w'$, not on $\boldsymbol{t}$ or the specific representative $u_{\boldsymbol{t}}$ of the coset $w\Lambda$. Crucially, the distribution of $D_w^{w'}$ also depends only on $w$ and $w'$. Note that 
\begin{equation}\label{eq:Dbounded}
\norm{D_w^{w'}}\leq T_w^{w'}\norm{\theta^\vee}. 
\end{equation}

A more intuitive description of $D_w^{w'}$ is as follows. For $u\in W$, the unique coroot vector contained in $\AA u$ is $\xx(u)$, so we can think of $\xx(u)$ as an approximation of the vector $u^\bullet$. Suppose we start with an arbitrary element $u\in w\Lambda$ and then repeatedly multiply on the left by independent copies of $v_p(\bb)$ until reaching an element $u'\in w'\Lambda$ (even if $w=w'$, we must multiply by at least one copy of $v_p(\bb)$). The number of copies of $v_p(\bb)$ that we use has the same distribution as $T_w^{w'}$, and $\xx(u')-\xx(u)$, which is the total displacement from $\xx(u)$ to $\xx(u')$, has the same distribution as $D_{w}^{w'}$. 

For all $w,w',x\in\overline W$, we can apply the right action of $x$ on $V$ to obtain the random vector $D_w^{w'}x$, which has the same distribution as $D_{wx}^{w'x}$:
\begin{equation}\label{eq:group_translate_D}
D_{w}^{w'}x\sim D_{wx}^{w'x}. 
\end{equation}

Let us write $\mathcal L_\id=\{\KK_0<\KK_1<\cdots\}$. Note that $\KK_0=0$. Observe that \[\KK_1-\KK_0,\, \KK_2-\KK_1,\, \KK_3-\KK_2,\ldots\] are i.i.d.\ random variables with the same distribution as $T_\id^\id$. In particular, these random variables have mean $|\overline W|$ because the stationary distribution of $\mathscr{M}$ is uniform. Furthermore, \[\xx(u_{\KK_1})-\xx(u_{\KK_0}),\, \xx(u_{\KK_2})-\xx(u_{\KK_1}),\, \xx(u_{\KK_3})-\xx(u_{\KK_2}),\ldots\] are i.i.d.\ random variables with the same distribution as $D_\id^\id$. 

\begin{lemma}\label{lem:exponential_tail}
There is a constant $\lambda>0$ such that \[\PP\left[\norm{D_w^{w'}}\geq R\right]=O(e^{-\lambda R})\] for all $w,w'\in\overline W$ and all real numbers $R$.
\end{lemma}
\begin{proof}
Because $\mathscr{M}$ is an irreducible finite-state Markov chain, there is a constant $\lambda>0$ such that \[\PP\left[T_w^{w'}\geq R/\norm{\theta^\vee}\right]=O(e^{-\lambda R})\] for all $w,w'\in\overline W$. Therefore, the desired result follows from \eqref{eq:Dbounded}. 
\end{proof}

The preceding lemma implies that for all $w,w'\in\overline W$, the covariance matrix $\Cov(D_w^{w'})$ is finite. 

\begin{lemma}\label{lem:mean_0}
As $t\to\infty$, the random variable $\xx(u_{\KK_t})/\sqrt{t}$ converges in distribution to the multivariate normal distribution $\NN(0,\Cov(D_\id^\id))$. 
\end{lemma}
\begin{proof}
Choose $x\in\overline W$, and write $\mathcal{L}_x=\{L_0^x<L_1^x<\cdots\}$. We wish to show that $\EE[D_\id^\id]=\EE[D_x^x]$, so suppose by way of contradiction that this is not the case. Because \[\xx(u_{\KK_1})-\xx(u_{\KK_0}),\, \xx(u_{\KK_2})-\xx(u_{\KK_1}),\, \xx(u_{\KK_3})-\xx(u_{\KK_2}),\ldots\] are i.i.d.\ random variables with the same distribution as $D_\id^\id$ (which we have seen has finite covariance), the Central Limit Theorem tells us that $(\xx(u_{\KK_t})-t\EE[D_\id^\id])/\sqrt{t}$ converges in distribution to $\NN(0,\Cov(D_\id^\id))$. Similarly, \[\xx(u_{\KK_1^x})-\xx(u_{\KK_0^x}),\, \xx(u_{\KK_2^x})-\xx(u_{\KK_1^x}),\, \xx(u_{\KK_3^x})-\xx(u_{\KK_2^x}),\ldots\] are i.i.d.\ random variables with the same distribution as $D_x^x$, so $(\xx(u_{\KK_t^x})-t\EE[D_x^x])/\sqrt{t}$ converges in distribution to $\NN(0,\Cov(D_x^x))$. It follows that there are constants $\upsilon_x,\Upsilon_x>0$ such that with probability $1$, we have \[\norm{\xx\left(u_{\KK_t}\right)-\xx\left(u_{\KK_{t'}^x}\right)}>\upsilon_x\norm{t\EE[D_\id^\id]-t'\EE[D_x^x]}\] for all integers $t,t'\geq \Upsilon_x$. Since $\EE[D_\id^\id]\neq\EE[D_x^x]$, this implies that there is a constant $\upsilon'_x>0$ such that \[\norm{\xx\left(u_{\KK_t}\right)-\xx\left(u_{\KK_{t'}^x}\right)}>\upsilon'_xt\] for all integers $t,t'\geq \Upsilon_x$. In particular, this holds when $t'=\min\{i\in\ZZ_{\geq 0}:\KK_i^x\geq\KK_t\}$, in which case \[\norm{\xx\left(u_{\KK_t}\right)-\xx\left(u_{\KK_{t'}^x}\right)}\] has the same distribution as $\norm{D_\id^x}$. This contradicts \cref{lem:exponential_tail}. We conclude that $\EE[D_\id^\id]=\EE[D_x^x]$.

As $x$ was arbitrary, we conclude that $\EE[D_\id^\id]$ is fixed by the right action of $\overline W$ on $V$. Because $V$ is an irreducible representation of $\overline W$, we must have $\EE[D_\id^\id]=0$. This completes the proof. 
\end{proof}

For future reference, let us record the fact, which we observed in the proof of \cref{lem:mean_0}, that 
\begin{equation}\label{eq:ED11=0}
\EE[D_x^x]=0 
\end{equation}
for all $x\in\overline W$. 

For each integer $K\geq 0$, let \[\hh(K)=\min\{t\in\ZZ_{\geq 0} : \KK_t\geq K\}.\]  

\begin{lemma}\label{lem:o(1)}
As $K\to\infty$, we have 
\[\PP\left[\left|\hh(K)-\left\lfloor K/|\overline W|\right\rfloor\right|>K^{0.6}\right]=o(1).\]
\end{lemma}
\begin{proof}
By the Central Limit Theorem, the random variables $(\KK_k-k|\overline W|)/\sqrt{k}$ converge in distribution as $k\to\infty$ to a normal distribution $\NN(0,\varsigma^2)$ for some $\varsigma>0$. 
Since $h(K)\leq K$, we have
\[\PP\left[\left|\hh(K)-\KK_{\hh(K)}/|\overline W|\right|>K^{0.55}\right]\leq\PP\left[\left|\hh(K)-\KK_{\hh(K)}/|\overline W|\right|>\hh(K)^{0.55}\right]=o(1).\] The proof follows from the observation that 
\[\PP\left[\left|\KK_{\hh(K)}-K\right|>(\log K)^2\right]=o(1). \qedhere\]
\end{proof}

\begin{lemma}\label{lem:o(1)_again}
As $K\to\infty$, we have 
\[\PP\left[\norm{u_K^\bullet-\xx\left(u_{\KK_{\left\lfloor K/|\overline W|\right\rfloor}}\right)}>K^{0.49}\right]=o(1).\] 
\end{lemma} 

\begin{proof}
We have 
\[\PP\left[\norm{u_K^\bullet-\xx\left(u_{\KK_{\hh(K)}}\right)}>(\log K)^2\right]=o(1),\] so it suffices to show that 
\[\PP\left[\norm{\xx\left(u_{\KK_{\hh(K)}}\right)-\xx\left(u_{\KK_{\left\lfloor K/|\overline W|\right\rfloor}}\right)}>K^{0.48}\right]=o(1).\] Let $\boldsymbol{\Xi}$ be the set of pairs $(t,t')$ of integers satisfying \[\left\lfloor K/|\overline W|\right\rfloor-K^{0.6}\leq t<t'\leq \left\lfloor K/|\overline W|\right\rfloor+K^{0.6}.
\] 
Let $\mathtt{A}$ be the event that $\left|\hh(K)-\left\lfloor K/|\overline W|\right\rfloor\right|>K^{0.6}$, and let $\mathtt{B}$ be the event that there exists $(t,t')\in\boldsymbol{\Xi}$ such that $\norm{\xx\left(u_{\KK_{t'}}\right)-\xx\left(u_{\KK_{t}}\right)}>K^{0.48}$. If \[\norm{\xx\left(u_{\KK_{\hh(K)}}\right)-\xx\left(u_{\KK_{\left\lfloor K/|\overline W|\right\rfloor}}\right)}>K^{0.48},\] then $\mathtt{A}$ or $\mathtt{B}$ must occur. According to \cref{lem:o(1)}, $\PP[\mathtt{A}]=o(1)$. Hence, we just need to show that $\PP[\mathtt{B}]=o(1)$. 

For $j\in[r]$ and $(t,t')\in\boldsymbol{\Xi}$, let $\mathtt{B}_j(t,t')$ be the event that $\left|\left\langle \xx\left(u_{L_{t'}}\right)-\xx\left(u_{L_{t}}\right),\varepsilon_j\right\rangle\right|>K^{0.48}/r$. If $\mathtt{B}$ occurs, then one of the events $\mathtt{B}_j(t,t')$ must occur; indeed, this follows from the inequality 
\[\norm{\xx\left(u_{\KK_{t'}}\right)-\xx\left(u_{\KK_{t}}\right)}\leq\sum_{j=1}^r\left|\left\langle \xx\left(u_{L_{t'}}\right)-\xx\left(u_{L_{t}}\right),\varepsilon_j\right\rangle\right|.\] Since $|\boldsymbol{\Xi}|\leq 4K^{1.2}$, we just need to show that $\PP[\mathtt{B}_j(t,t')]=o(K^{-1.2})$ for each $j\in[r]$ and $(t,t')\in\boldsymbol{\Xi}$. 

Fix $j\in[r]$ and $(t,t')\in\boldsymbol{\Xi}$. Let \[G(z)=\EE\left[e^{\langle D_\id^\id,\varepsilon_j\rangle z}\right]\] be the moment generating function of the random variable $\langle D_\id^\id,\varepsilon_j\rangle$. For $z\neq 0$, we can use \cref{lem:exponential_tail} to find that 
\begin{align*}
G(z)&\leq\int_0^\infty\mathbb P\left[e^{\norm{D_\id^\id z}}\geq R\right]\mathrm{d}R \\ 
&=\int_0^\infty\mathbb P\left[\norm{D_\id^\id}\geq (\log R)/|z|\right]\mathrm{d}R \\ 
&\leq 1+O\left(\int_1^\infty R^{-\lambda/|z|}\mathrm{d}R\right).
\end{align*}
Hence, $G(z)<\infty$ whenever $|z|<\lambda/2$. We know by  \eqref{eq:ED11=0} that $\EE[\langle D_\id^\id,\varepsilon_j\rangle]=0$. Therefore, we can use the Taylor expansion of $G$ at $0$ to find that there is a constant $C>0$ (independent of $K$) such that $G(z)<1+Cz^2$ whenever $|z|<\lambda/2$. The random variable $\left\langle \xx\left(u_{L_{t'}}\right)-\xx\left(u_{L_{t}}\right),\varepsilon_j\right\rangle$ is a sum of $t'-t$ i.i.d.\ random variables with the same distribution as $\langle D_\id^\id,\varepsilon_j\rangle$, so its moment generating function is $G(z)^{t'-t}$. It follows from a standard Chernoff bound that 
\begin{align*}
\PP[\mathtt{B}_j(t,t')]&=\PP\left[\left|\left\langle \xx\left(u_{L_{t'}}\right)-\xx\left(u_{L_{t}}\right),\varepsilon_j\right\rangle\right|>K^{0.48}/r\right] \\ 
&\leq(1+Cz^2)^{t'-t}e^{-zK^{0.48}/r}
\\ 
&\leq(1+Cz^2)^{2K^{0.6}}e^{-zK^{0.48}/r}
\end{align*} 
for all $z\in(0,\lambda/2)$. If $K$ is sufficiently large, then we can set $z=K^{-0.2}$ to find that 
\begin{align*}
\PP[\mathtt{B}_j(t,t')]&\leq(1+CK^{-0.4})^{2K^{0.6}}e^{-K^{0.28}/r} \\ 
&\leq (e^{CK^{-0.4}})^{2K^{0.6}}e^{-K^{0.28}/r} \\ 
&=o(K^{-1.2}),
\end{align*} 
as desired. 
\end{proof}

We are now in a position to prove \cref{thm:normal}, which states that $u_K^\bullet/\sqrt{K}$ converges in distribution to a multivariate normal distribution of the form $\NN(0,\sigma^2_\bb \II_r)$ for some $\sigma_\bb>0$. 

\begin{proof}[Proof of \cref{thm:normal}]
Let $\boldsymbol{\Sigma}=\frac{1}{|\overline W|}\Cov{D_\id^\id}$. It follows from \cref{lem:mean_0,lem:o(1)_again} that $u_K^\bullet/\sqrt{K}$ converges in distribution to $\NN(0,\boldsymbol{\Sigma})$. Choose $x\in\overline W$, and let $k_x=\min\mathcal{L}_x$. Let \[\widetilde u_{K}=v^{(k_x+K)}v^{(k_x+K-1)}\cdots v^{(k_x+1)}.\] Then $\widetilde u_K$ has the same distribution as $u_K$, so $\widetilde u_K^\bullet/\sqrt{K}$ converges in distribution to $\NN(0,\boldsymbol{\Sigma})$. But $\AA u_{k_x+K}=\AA\widetilde u_Ku_{k_x}=\AA \widetilde u_Kx+\xx(u_{k_x})$, so \[\norm{(k_x+K)^{-1/2}u_{k_x+K}^\bullet-K^{-1/2}u_{K}^\bullet x}\to 0\] as $K\to\infty$. This implies that $(k_x+K)^{-1/2}u_{k_x+K}^\bullet$ converges in distribution to $\mathcal{N}(0, P_x \boldsymbol{\Sigma} P_x^{-1})$. Hence, 
\[
\boldsymbol{\Sigma} = P_x \boldsymbol{\Sigma} P_x^{-1}.
\]
As $x$ was arbitrary, we deduce that $\boldsymbol{\Sigma}$ commutes with every matrix of the form $P_w$ for $w\in\overline W$. As $V$ is an irreducible representation of $\overline{W}$, we can use Schur's Lemma to deduce that $\boldsymbol{\Sigma}$ is of the form $\sigma_\bb^2\II_r$ for some $\sigma_\bb\geq 0$. Because $\boldsymbol{\Sigma}=\frac{1}{|\overline W|}\Cov{D_\id^\id}$, we have 
\begin{equation}\label{eq:sigma_b_trace}
\sigma_\bb^2=\frac{1}{r|\overline W|}\Tr(\Cov(D_{\id}^{\id})). 
\end{equation}
This implies that $\sigma_\bb>0$. 
\end{proof}

\begin{proof}[Proof of \cref{cor:recurrent}]
Suppose first that $r \le 2$. Since
\[
\xx(u_{\KK_1})-\xx(u_{\KK_0}),\, \xx(u_{\KK_2})-\xx(u_{\KK_1}),\, \xx(u_{\KK_3})-\xx(u_{\KK_2}),\ldots
\]
are i.i.d.\ random vectors of mean $0$ in the $r$-dimensional lattice $Q^\vee$, it follows from the classical theorem on recurrence of random walks \cite[Section~7]{Spi76} that there are infinitely many positive integers $M$ for which
\[
\xx(u_{\KK_M})=\xx(u_{\KK_M}) - \xx(u_{\KK_0}) = \sum_{i=1}^M \left( \xx(u_{\KK_i}) - \xx(u_{\KK_{i-1}}) \right) = 0.
\]
For such $M$, the definition of $\KK_M$ ensures that $\overline u_{\KK_M}=\id$, so $u_{\KK_M}=\id$.

Now suppose $r\geq 3$. To show that the random billiard walk is transient, it suffices by a classical result of P\'olya (see \cite[Section~2]{Spi76}) to show that \begin{equation}\label{eq:Polya}
\sum_{K\geq 0}\PP[u_K=\id]<\infty.
\end{equation} Let $\boldsymbol{\mu}$ be the probability distribution on $W$ defined by $\boldsymbol{\mu}(u)=\PP[v_p(\bb)=u]$. The sequence $(u_K)_{K\geq 0}$ is a random walk on $W$ driven by $\boldsymbol{\mu}$. It follows from \cref{thm:normal} that $\lim_{K\to 0}\norm{u_K^\bullet}/K=0$ almost surely, so $\boldsymbol{\mu}$ is centered. Since $W$ has polynomial growth of degree $r$, it follows from \cite[Corollary~1.17]{Alexopoulos} that there are constants $c_1,c_2>0$ such that \[\left\lvert\PP[u_K=\id]-c_1K^{-r/2}\right\rvert<c_2K^{-(r+1)/2}\] for all $K\geq 1$. 
This readily implies \eqref{eq:Polya}. 
\end{proof}

We can now prove \cref{cor:length}, which provides an asymptotic formula for the expected Coxeter length $\EE[\ell(v_p(\bb^K))]$. 

\begin{proof}[Proof of \cref{cor:length}]
For $\beta\in\Phi^+$ and $u\in W$, let $\ell_\beta(u)$ denote the number of hyperplanes of the form $\H_\beta^k$ that separate the centroid of the alcove $\AA u$ from the centroid of the alcove $\AA$. Then $\ell(u)=\sum_{\beta\in\Phi^+}\ell_\beta(u)$. We have $\ell_\beta(u)=|\langle u^\bullet,\beta\rangle|+O(1)$. Hence, 
\[
\EE[\ell(v_p(\bb^K))]=\sum_{\beta\in\Phi^+}\EE\left[\left\lvert\langle v_p(\bb^K)^\bullet/\sqrt{K},\beta\rangle\right\rvert\right]\sqrt{K}+O(1).
\]
It follows from \cref{thm:normal} that for each fixed $\beta\in\Phi^+$, the random variable $\langle v_p(\bb^K)^\bullet/\sqrt{K},\beta\rangle$ converges in distribution to a (univariate) normal distribution with mean $0$ and variance $\sigma_\bb^2\norm{\beta}^2$. Consequently,
\[\lim_{K\to\infty}\EE\left[\left\lvert\langle v_p(\bb^K)^\bullet/\sqrt{K},\beta\rangle\right\rvert\right]=\frac{1}{\sqrt{2\pi}\sigma_\bb\norm{\beta}}\int_{-\infty}^\infty |x|e^{-x^2/(2\sigma_\bb^2\norm{\beta}^2)}\mathrm{d}x=\sqrt{\frac{2}{\pi}}\sigma_\bb\norm{\beta},\]
and this implies the desired result. 
\end{proof} 

\section{Computing the Covariance}\label{sec:computing}
\subsection{Invariance of the Covariance}\label{subsec:invariance}  
Our goal in this section is to compute the number $\sigma_\bb^2$ appearing in \cref{thm:normal}. As an initial step, we prove a simple lemma providing conditions under which we can change the word $\bb$ without changing the number $\sigma_\bb^2$. 

We can apply a \dfn{commutation move} to a word over $S$ by swapping two adjacent simple reflections that commute with each other in $W$. Let us say two finite words $\aa$ and $\aa'$ over $S$ are \dfn{cyclically commutation equivalent} if $\aa'$ can be obtained from $\aa$ by applying a sequence of commutation moves and cyclic rotations. 

\begin{lemma}\label{lem:invariance}
If $\aa$ and $\aa'$ are cyclically commutation equivalent finite words over $S$, then $\sigma_\aa^2=\sigma_{\aa'}^2$. 
\end{lemma} 

\begin{proof}
It suffices to prove that $\sigma_{\aa}^2=\sigma_{\aa'}^2$ when $\aa$ is obtained from $\aa'$ by performing a single commutation move or by cyclically rotating by one letter. If $\aa'$ is obtained from $\aa$ by applying a commutation move, then it is straightforward to check that $v_p(\aa^K)$ has the same distribution as $v_p((\aa')^K)$ for each $K\geq 0$, so the desired equality is immediate in this case. Now suppose $\aa$ is obtained from $\aa'$ by cyclically rotating by one letter. Consider a large positive integer $K$. The word $\aa^K$ is obtained by cyclically rotating $(\aa')^K$ by one letter. There is an obvious coupling of $\sub_p(\aa^K)$ and $\sub_p((\aa')^K)$, where a letter in $\aa^K$ appears in $\sub_p(\aa^K)$ if and only if the corresponding letter of $(\aa')^K$ appears in $\sub_p((\aa')^K)$. With this coupling, the Coxeter lengths of $v_p(\aa^K)$ and $v_p((\aa')^K)$ differ by at most $2$. Therefore, the desired equality follows from \cref{cor:length}, which implies that $\sigma_\aa^2$ and $\sigma_{\aa'}^2$ are determined by the asymptotics of the expected values of these Coxeter lengths as $K\to\infty$. 
\end{proof}

\subsection{An Invertible Transformation} 

Let us quickly prove the following lemma, which is necessary to define the vector $\kappa_\bb$ appearing in the statement of \cref{thm:covariance_complicated}. 

\begin{lemma}\label{lem:invertible}
Let $\bb$ be a finite word over $S$ that contains each simple reflection at least once. The linear transformation $\II_V-R_\bb$ of $V$ is invertible. 
\end{lemma} 
\begin{proof}
Let $\bb=s_{i_{m}}\cdots s_{i_1}$, and for ${J=\{j_1<\cdots<j_k\}\subseteq[m]}$, let $\overline s_J=\overline s_{i_{j_k}}\cdots\overline s_{i_{j_1}}$. Let $\gamma\in V$ be such that $(\II_V-R_\bb)\gamma=0$; we will prove that $\gamma=0$. For each $w\in\overline W$, the linear transformation $P_w$ is orthogonal, so $\norm{P_w\gamma}=\norm{\gamma}$. Hence,  
\begin{align*}
\norm{\gamma}&=\norm{R_\bb\gamma} \\ 
&=\norm{\sum_{J\subseteq[m]}p^{|J|}(1-p)^{m-|J|}P_{\overline{s}_J}\gamma} \\ 
&\leq \sum_{J\subseteq[m]}p^{|J|}(1-p)^{m-|J|}\norm{P_{\overline{s}_J}\gamma} \\ 
&= \sum_{J\subseteq[m]}p^{|J|}(1-p)^{m-|J|}\norm{\gamma} \\ 
&=\norm{\gamma}. 
\end{align*}
This inequality must be an equality, so all of the vectors $P_{\overline s_J}\gamma$ for $J\subseteq[m]$ must be equal. Since every simple reflection appears in $\bb$, this implies that $\gamma=P_{\overline s}\gamma$ for every $s\in S$. Consequently, $\gamma=P_{w}\gamma$ for every $w\in\overline W$. As $V$ is an irreducible representation of $\overline W$, we deduce that $\gamma=0$. 
\end{proof}

\subsection{A General Formula} 
We now wish to prove \cref{thm:covariance_complicated}, which provides a general formula for computing $\sigma_\bb^2$. If $\bb$ does not contain all of the simple reflections in $S$, then the subgroup of $W$ generated by the simple reflections in $\bb$ is finite, so $\sigma_\bb^2=0$. Hence, we will assume that $\bb$ contains each simple reflection at least once. 

Recall that we write $\bb=s_{i_m}\cdots s_{i_1}$ and let 
\begin{align*}
\Z&=\{j\in[m]:i_j=0\}, \\ 
\GG_J&=\sum_{k\in J\cap\Z}{P}_{\overline s_{J\cap[k-1]}}\theta^\vee, \\ 
\chi_\bb&=\sum_{J\subseteq[m]}p^{|J|}(1-p)^{m-|J|}\norm{\GG_J}^2, \\ 
\ED_\bb&=|\overline W|(\II_V-R_\bb)^{-1}\sum_{J\subseteq[m]}p^{|J|}(1-p)^{m-|J|}\GG_J. 
\end{align*}

\begin{proposition}\label{Prop:Master} 
Let $U$ be a real vector space, and let $f\colon V\to U$ be a function. For all $x\in \overline W$, we have  
\[\EE[f(D_\id^{x})]=\sum_{J\subseteq[m]}p^{|J|}(1-p)^{m-|J|}\EE[f(\delta_{\overline s_J\neq x}D_{\overline s_J}^x+\GG_J)].\] 
\end{proposition}

\begin{proof}
Note that $u_1=s_{J_0}$, where $J_0\subseteq[m]$ is obtained by including each element of $[m]$ in $J_0$
independently with probability $p$. For each $J\subseteq[m]$, we have $\PP[J_0=J]=p^{|J|}(1-p)^{m-|J|}$. The desired result follows from the observation that $\xi(u_1)=\nu_{J_0}$. 
\end{proof} 

\begin{lemma}\label{lem:sum_E}
We have 
\[\sum_{x\in\overline W}\EE[D_{\id}^x]=\ED_\bb.\] 
\end{lemma}
\begin{proof}
Applying \cref{Prop:Master} with $U=V$ and the identity map $f\colon V\to V$ yields that 
\begin{align*}
\sum_{x\in\overline W}\EE[D_\id^x]&=\sum_{x\in\overline W}\sum_{J\subseteq [m]}p^{|J|}(1-p)^{m-|J|}
\EE[\delta_{\overline s_J\neq x}D_{\overline s_J}^x+\nu_J] \\ 
&=\sum_{J\subseteq [m]}p^{|J|}(1-p)^{m-|J|}
\left(\sum_{x\in\overline W}\EE[\delta_{\overline s_J\neq x}D_{\overline s_J}^x]+|\overline W|\nu_J\right). 
\end{align*}
For every $J\subseteq[m]$, it follows from \eqref{eq:group_translate_D} and \eqref{eq:ED11=0} that 
\begin{align*}
\sum_{x\in\overline W}\EE[\delta_{\overline s_J\neq x}D_{\overline s_J}^x]&=\sum_{x\in\overline W}\EE[D_{\overline s_J}^x] \\ 
&=\sum_{x\in\overline W}\EE[D_{\overline s_J}^{x\overline s_J}] \\ 
&=P_{\overline s_J}\sum_{x\in\overline W}\EE[D_{\id}^{x}].
\end{align*}
We conclude that 
\begin{align*}
\sum_{x\in\overline W}\EE[D_{\id}^x]&=\sum_{J\subseteq[m]}p^{|J|}(1-p)^{m-|J|}P_{\overline s_J}\sum_{x\in \overline W}\EE[D_\id^x]+|\overline W|\sum_{J\subseteq[m]}p^{|J|}(1-p)^{m-|J|}\nu_J \\ 
&=R_\bb\sum_{x\in \overline W}\EE[D_\id^x]+|\overline W|\sum_{J\subseteq[m]}p^{|J|}(1-p)^{m-|J|}\nu_J, 
\end{align*}
so 
\begin{align*}\sum_{x\in \overline W}\EE[D_\id^x]&=|\overline W|(\II_V-R_\bb)^{-1}\sum_{J\subseteq[m]}p^{|J|}(1-p)^{m-|J|}\nu_J \\ 
&=\kappa_\bb,
\end{align*}
as desired. 
\end{proof} 

For $w,x\in\overline W$, let \[\Gamma_x^{w}=\EE[\Tr(D_x^{w}(D_x^w)^\top)].\] We know by \eqref{eq:ED11=0} that $\EE[D_\id^\id]=0$, so we can use \eqref{eq:sigma_b_trace} to see that 
\begin{equation}\label{eq:sigma_Gamma}
\sigma_\bb^2=\frac{1}{r|\overline W|}\Tr(\Cov(D_\id^\id))=\frac{1}{r|\overline W|}\Gamma_\id^\id.
\end{equation} 

\begin{proof}[Proof of \cref{thm:covariance_complicated}]
For $w,x\in\overline W$, \eqref{eq:group_translate_D} tells us that $D_x^w$ has the same distribution as ${P}_xD_\id^{wx^{-1}}$, so 
\begin{align}
\nonumber \Gamma_x^w&=\EE[\Tr(D_x^{w}(D_x^w)^\top)] \\ 
\nonumber &=\EE[\Tr(({P}_xD_\id^{wx^{-1}})({P}_xD_\id^{wx^{-1}})^\top)] \\ 
\nonumber &=\EE[\Tr({P}_xD_\id^{wx^{-1}}(D_\id^{wx^{-1}})^\top{P}_x^{-1})] \\ 
\nonumber &=\EE[\Tr(D_\id^{wx^{-1}}(D_\id^{wx^{-1}})^\top)] \\ 
&=\Gamma_{\id}^{xw^{-1}}.
\label{eq:translate_Gamma}
\end{align}

Let $U=\mathbb R^{r\times r}$, and define $f\colon V\to U$ by $f(\gamma)=\gamma\gamma^\top$. Invoking \cref{Prop:Master}, we find that $\EE[D_\id^x(D_\id^x)^\top]$ is equal to
\[
\sum_{J\subseteq[m]}p^{|J|}(1-p)^{m-|J|}\left[\delta_{\overline s_J=x}\EE[D_{\overline s_J}^x(D_{\overline s_J}^x)^\top]+\delta_{\overline s_J=x}\EE[D_{\overline s_J}^x]\GG_J^\top+\delta_{\overline s_J=x}\EE[(D_{\overline s_J}^x)^\top]\GG_J+\GG_J\GG_J^\top\right].
\] Taking the trace and applying \eqref{eq:translate_Gamma} yields that 
\begin{align}
\nonumber \Gamma_\id^x&=\!\sum_{J\subseteq[m]}p^{|J|}(1-p)^{m-|J|}\!\left[\delta_{\overline s_J=x}\Gamma_{\id}^{x\overline s_J^{-1}}\!\!+\!\delta_{\overline s_J=x}\Tr(\EE[D_{\overline s_J}^x]\GG_J^\top)\!+\!\delta_{\overline s_J=x}\Tr(\EE[(D_{\overline s_J}^x)^\top]\GG_J)\!+\!\Tr(\GG_J\GG_J^\top)\right] \\ 
\nonumber &=\!\sum_{J\subseteq[m]}p^{|J|}(1-p)^{m-|J|}\!\left[\delta_{\overline s_J=x}\Gamma_{\id}^{x\overline s_J^{-1}}\!\!+2\delta_{\overline s_J=x}\Tr(\EE[D_{\overline s_J}^x]\GG_J^\top)+\langle\GG_J,\GG_J\rangle\right] \\ 
&\label{eq:computations1}=\!\sum_{J\subseteq[m]}p^{|J|}(1-p)^{m-|J|}\!\left[\delta_{\overline s_J=x}\Gamma_{\id}^{x\overline s_J^{-1}}\!\!+2\delta_{\overline s_J=x}\Tr(\EE[D_{\overline s_J}^x]\GG_J^\top)\right]+\chi_\bb. 
\end{align} 
The key trick is now to sum \eqref{eq:computations1} over $x\in\overline W$: 
\begin{align*}
\sum_{x\in\overline W}\Gamma_\id^x&=\sum_{x\in\overline W}\left(\sum_{J\subseteq[m]}p^{|J|}(1-p)^{m-|J|}\left[\delta_{\overline s_J=x}\Gamma_{\id}^{x\overline s_J^{-1}}+2\delta_{\overline s_J=x}\Tr(\EE[D_{\overline s_J}^x]\GG_J^\top)\right]+\chi_\bb\right) \\ &=\sum_{J\subseteq[m]}p^{|J|}(1-p)^{m-|J|}\left[\left(\sum_{x\in\overline W}\Gamma_{\id}^{x\overline s_J^{-1}}-\Gamma_{\id}^{\id}\right)+2\sum_{x\in\overline W\setminus\{\overline s_J\}}\Tr(\EE[D_{\overline s_J}^x]\GG_J^\top)\right]+|\overline W|\chi_\bb. 
\end{align*}
Since $\EE[D_{\overline s_J}^{\overline s_J}]=0$ by \eqref{eq:ED11=0}, we have \[\sum_{x\in\overline W\setminus\{\overline s_J\}}\Tr(\EE[D_{\overline{s}_J}^x]\GG_J^\top)=\sum_{x\in\overline W}\Tr(\EE[D_{\overline{s}_J}^x]\GG_J^\top).\] Also, $\sum_{x\in\overline W}\Gamma_{\id}^{x\overline s_J^{-1}}=\sum_{x\in\overline W}\Gamma_{\id}^x$. Therefore, 
\begin{align*}
\sum_{x\in\overline W}\Gamma_{\id}^x&=\sum_{J\subseteq[m]}p^{|J|}(1-p)^{m-|J|}\left[\left(\sum_{x\in\overline W}\Gamma_{\id}^{x}-\Gamma_{\id}^{\id}\right)+2\sum_{x\in\overline W}\Tr(\EE[D_{\overline s_J}^x]\GG_J^\top)\right]+|\overline W|\chi_\bb \\ 
&= \sum_{x\in\overline W}\Gamma_{\id}^{x}-\Gamma_{\id}^{\id}+2\sum_{J\subseteq[m]}p^{|J|}(1-p)^{m-|J|}\sum_{x\in\overline W}\Tr(\EE[D_{\overline s_J}^x]\GG_J^\top)+|\overline W|\chi_\bb \\ 
&= \sum_{x\in\overline W}\Gamma_{\id}^{x}-\Gamma_{\id}^{\id}+2\sum_{J\subseteq[m]}p^{|J|}(1-p)^{m-|J|}\Tr\left(\sum_{x\in\overline W}\EE\left[{P}_{\overline s_J}D_{\id}^{x\overline s_J^{-1}}\right]\GG_J^\top\right)+|\overline W|\chi_\bb \\ 
&= \sum_{x\in\overline W}\Gamma_{\id}^{x}-\Gamma_{\id}^{\id}+2\sum_{J\subseteq[m]}p^{|J|}(1-p)^{m-|J|}\Tr\left({P}_{\overline s_J}\sum_{x\in\overline W}\EE\left[D_{\id}^{x}\right]\GG_J^\top\right)+|\overline W|\chi_\bb \\ 
&= \sum_{x\in\overline W}\Gamma_{\id}^{x}-\Gamma_{\id}^{\id}+2\sum_{J\subseteq[m]}p^{|J|}(1-p)^{m-|J|}\Tr\left({P}_{\overline s_J}\ED_\bb\GG_J^\top\right)+|\overline W|\chi_\bb,
\end{align*}
where the last equation follows from \cref{lem:sum_E}. 
Thankfully, we can now cancel $\sum_{x\in\overline W}\Gamma_{\id}^x$ to find that 
\begin{align*} 
\Gamma_{\id}^{\id}&=2\sum_{J\subseteq[m]}p^{|J|}(1-p)^{m-|J|}\Tr\left({P}_{\overline s_J}\ED_\bb\GG_J^\top\right)+|\overline W|\chi_\bb \\ 
&=2\sum_{J\subseteq[m]}p^{|J|}(1-p)^{m-|J|}\langle\GG_J,{P}_{\overline s_J}\ED_\bb\rangle+|\overline W|\chi_\bb.
\end{align*}
The formula for $\sigma_\bb^2$ in the statement of the theorem now follows from \eqref{eq:sigma_Gamma}.  
\end{proof} 

\begin{remark}\label{rem:continuous}
For each fixed $J\subseteq[m]$, it follows from the definitions of $\GG_J$, ${P}_{\overline s_J}$, and $\ED_\bb$ that $\langle\GG_J,{P}_{\overline s_J}\ED_\bb\rangle$ is a rational function in $p$ for $p\in(0,1)$. For each fixed $p\in(0,1)$, it follows from \cref{lem:sum_E} that $\langle\GG_J,{P}_{\overline s_J}\ED_\bb\rangle$ is finite. Hence, $\langle\GG_J,{P}_{\overline s_J}\ED_\bb\rangle$ is continuous in $p$ for $p\in(0,1)$. It is also clear that $\chi_\bb$ is continuous in $p$. Therefore, it follows from \cref{thm:covariance_complicated} that $\sigma_\bb^2$ is a continuous function of $p$ for $p\in(0,1)$. 
\end{remark} 

\subsection{A Simpler Formula} 

We can now proceed to prove \cref{thm:covariance_simple}, which gives a very simple formula for $\sigma_\bb$ when $\bb$ contains exactly one occurrence of the simple reflection $s_0$. 

\begin{proof}[Proof of \cref{thm:covariance_simple}]
By \cref{lem:invariance}, it suffices to prove the theorem in the case where $\bb=\aa$ (i.e., the rightmost letter of $\bb$ is $s_0$). Under this assumption, we have $R_\aa=R_\bb$ and $\Z=\{1\}$. For $J\subseteq[m]$, we have $\GG_J=\delta_{1\in J}\theta^\vee$. Hence, 
\[\chi_\bb=\sum_{J\subseteq[m]}p^{|J|}(1-p)^{m-|J|}\delta_{1\in J}\norm{\theta^\vee}^2=p\norm{\theta^\vee}^2. \]
Moreover, 
\begin{equation}\label{eq:varrho_simplified}
\ED_\bb=|\overline W|(\II_V-R_\bb)^{-1}\sum_{J\subseteq[m]}p^{|J|}(1-p)^{m-|J|}\delta_{1\in J}\theta^\vee=p|\overline W|(\II_V-R_\bb)^{-1}\theta^\vee. 
\end{equation} 
For $J\subseteq[m]$, let $F(J)=p^{|J|}(1-p)^{m-|J|}\langle\theta^\vee,{P}_{\overline s_J}\ED_\bb\rangle$. According to \cref{thm:covariance_complicated}, we have 
\begin{equation}\label{eq:sigma_in_proof}
\sigma_\bb^2=\frac{1}{r}\left(\frac{2}{|\overline W|}\sum_{1\in J\subseteq [m]}F(J)+p\norm{\theta^\vee}^2\right). 
\end{equation} 

We know by \cref{rem:continuous} that $\sigma_\bb^2$ is a continuous function of $p$ for $p\in(0,1)$. Therefore, in order to prove the theorem, it suffices to prove it under the assumption that $p\neq 1/2$; we will make this assumption in what follows. 

We will use a trick to simplify \eqref{eq:sigma_in_proof} further. Observe that 
\begin{align*}
\sum_{1\in J\subseteq[m]}F(J)&=\sum_{1\in J\subseteq[m]}p^{|J|}(1-p)^{m-|J|}\langle\theta^\vee,{P}_{\overline s_J}\ED_\bb\rangle \\ 
&=\sum_{1\not\in J'\subseteq[m]}p^{|J'|+1}(1-p)^{m-|J'|-1}\langle\theta^\vee,{P}_{\overline s_0}{P}_{\overline s_{J'}}\ED_\bb\rangle \\ 
&=\frac{p}{1-p}\sum_{1\not\in J'\subseteq[m]}p^{|J'|}(1-p)^{m-|J'|}\langle{P}_{\overline s_0}\theta^\vee,{P}_{\overline s_{J'}}\ED_\bb\rangle,
\end{align*} where we have used the fact that ${P}_{\overline s_0}$ is an orthogonal involution. Now, ${P}_{\overline s_0}\theta^\vee$ is the vector obtained by reflecting $\theta^\vee$ through the linear hyperplane $\H_{\theta}^0$, which is orthogonal to $\theta^\vee$. Hence, ${P}_{\overline s_0}\theta^\vee=-\theta^\vee$. It follows that 
\begin{align*}
\sum_{1\in J\subseteq[m]}F(J)&=-\frac{p}{1-p}\sum_{1\not\in J'\subseteq[m]}p^{|J'|}(1-p)^{m-|J'|}\langle\theta^\vee,{P}_{\overline s_{J'}}\ED_\bb\rangle \\ 
&=-\frac{p}{1-p}\sum_{1\not\in J'\subseteq[m]}F(J'),  
\end{align*}
so 
\begin{align*}\sum_{1\in J\subseteq[m]}F(J)&=\frac{p}{2p-1}\left(\sum_{1\in J\subseteq [m]}F(J)-\frac{1-p}{p}\sum_{1\in J\subseteq[m]}F(J)\right) \\ 
&=\frac{p}{2p-1}\left(\sum_{1\in J\subseteq [m]}F(J)+\sum_{1\not\in J'\subseteq[m]}F(J')\right) \\ 
&=\frac{p}{2p-1}\sum_{J\subseteq[m]}F(J) \\ 
&=\frac{p}{2p-1}\left\langle\theta^\vee,\sum_{J\subseteq[m]}p^{|J|}(1-p)^{m-|J|}{P}_{\overline s_J}\ED_\bb\right\rangle \\ 
&=\frac{p}{2p-1}\left\langle\theta^\vee,R_{\bb}\ED_\bb\right\rangle.
\end{align*} We can now use \eqref{eq:varrho_simplified} to see that $R_\bb\ED_\bb=\ED_\bb-(\II_V-R_\bb)\ED_\bb=\ED_\bb-p|\overline W|\theta^\vee$. Hence, 
\[\sum_{1\in J\subseteq[m]}F(J)=\frac{p}{2p-1}\left(\langle\theta^\vee,\ED_\bb\rangle-p|\overline W|\norm{\theta^\vee}^2\right).\] Substituting this into \eqref{eq:sigma_in_proof} and using the formula for $\ED_\bb$ in \eqref{eq:varrho_simplified} yields that 
\begin{align*}\sigma_\bb^2&=\frac{1}{r}\left(\frac{2}{|\overline W|}\frac{p}{2p-1}\left(\langle\theta^\vee,\ED_\bb\rangle-p|\overline W|\norm{\theta^\vee}^2\right)+p\norm{\theta^\vee}^2\right) \\ 
&=\frac{1}{r}\left(\frac{2}{|\overline W|}\frac{p}{2p-1}\langle\theta^\vee,p|\overline W|(\II_V-R_\bb)^{-1}\theta^\vee\rangle-\frac{p}{2p-1}\norm{\theta^\vee}^2\right) \\ 
&=\frac{1}{r}\frac{p}{2p-1}\langle\theta^\vee,(2p(\II_V-R_\bb)^{-1}-\II_V)\theta^\vee\rangle, 
\end{align*}
as desired. 
\end{proof}

\section{Explicit Computations}\label{sec:explicit}

Assume throughout this section that $\bb$ is a Coxeter word. We aim to prove \cref{thm:table}, which gives explicit formulas for $\sigma_\bb^2$ and $\lim_{K\to\infty}\EE[\ell(u_K)]/\sqrt{K}$. Throughout this section, let $e_i$ denote the $i$-th standard basis vector of the Euclidean space $\mathbb R^n$. We will view $V$ as a certain $r$-dimensional Euclidean subspace of $\mathbb R^n$. 

We consider the different families of affine Weyl groups separately. We have chosen to present the four infinite families in the order $\widetilde C,\widetilde B,\widetilde D,\widetilde A$, which might appear strange since it is not alphabetical. Our rationale is that we believe this is order presents the computations in roughly increasing order of complexity. 

\begin{remark}\label{rem:flip}
Given a Coxeter word $\bb$, we obtain an acyclic orientation $\mathrm{ao}(\bb)$ of the Coxeter graph of $W$ by directing each edge $\{s,s'\}$ from $s$ to $s'$ if $s$ appears to the left of $s'$ in $\bb$. We can perform a \dfn{flip} on an acyclic orientation of a graph by reversing the directions of all edges incident to a source or a sink. Two acyclic orientations are \dfn{flip equivalent} if one can be obtained from the other by a sequence of flips. It is known \cite[Theorem~1.15]{DMR} that two Coxeter words $\bb$ and $\bb'$ are cyclically commutation equivalent if and only if $\mathrm{ao}(\bb)$ and $\mathrm{ao}(\bb')$ are flip equivalent. It is also well known that all acyclic orientations of a tree are flip equivalent. When $W$ is not of type $\widetilde A$, its Coxeter graph is a tree, so we only need to consider one particular Coxeter word $\bb$. When $W$ is of type $\widetilde A_r$ for $r\geq 2$, its Coxeter graph is a cycle. Luckily, the flip equivalence classes of acyclic orientations of a cycle are easy to describe. Namely, if the cycle is embedded in the plane, then two acyclic orientations are flip equivalent if and only if they have the same number of edges directed clockwise.  
\end{remark}

\subsection{Type \texorpdfstring{$\widetilde C_r$}{widetilde{C}_r}} 
Let $\Phi$ be the root system of type $C_r$, where $r=n\geq 2$. That is, 
\[\Phi=\{\pm e_i\pm e_j:1\leq i<j\leq n\}\cup\{\pm 2e_i:i\in[n]\}.\] Then 
$V=\mathbb R^n$. For $1\leq i\leq n-1$, the $i$-th simple root is $\alpha_i=e_i-e_{i+1}$. Also, $\alpha_n=2e_n$. We have \[\theta^\vee=e_1.\] 
The Coxeter graph of the affine Weyl group $W=\widetilde C_n$ is 
\[\begin{array}{l}\includegraphics[height=0.448cm]{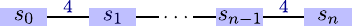}\end{array}.\] 

For $1\leq i\leq n-1$, the linear transformation $P_{s_i}$ swaps the vectors $e_i$ and $e_{i+1}$ and fixes $e_j$ for all ${j\in[n]\setminus\{i,i+1\}}$.  We have $P_{s_n}e_n=-e_n$ and $P_{s_n}e_j=e_j$ for all $j\in[n-1]$. In addition, we have $P_{\overline s_0}e_1=-e_1$ and $P_{\overline s_0}e_j=e_j$ for all $j\in[2,n]$.

Let $\cc=s_n\cdots s_1s_0$. The Coxeter graph of $\widetilde C_n$ is a tree, so \cref{rem:flip} tells us that to prove the formula for $\sigma_\bb^2$ in \cref{table:sigma_b}, it suffices to prove it when $\bb=\cc$. 

For $1\leq k\leq n+1$, let $\aa_k=s_n\cdots s_{n-k+1}$ be the prefix of $\cc$ of length $k$. Let 
\[\zeta=\frac{1}{4np(1-p)}\sum_{i=1}^n((2n-2i+1)(1-p)+p)e_i.\] It is straightforward to prove by induction on $k$ that 
\[R_{\aa_k}\zeta=\zeta-\frac{1}{2n(1-p)}e_{n-k+1}\] for all $1\leq k\leq n$. Therefore, 
\begin{align*}
R_{\cc}\zeta&=R_{s_0}R_{\aa_n}\zeta \\ 
&=R_{s_0}\left(\zeta-\frac{1}{2n(1-p)}e_1\right) \\ 
&=\zeta-e_1.
\end{align*} 
Consequently, 
\[(\II_V-R_\cc)\zeta=e_1=\theta^\vee.\] 

The desired formula for $\sigma_\cc^2$ now follows from \cref{thm:covariance_simple}. Indeed, for $p\neq 1/2$, we have 
\begin{align*}
\sigma_{\cc}^2&=\frac{1}{n}\frac{p}{2p-1}\langle\theta^\vee,(2p(\II_V-R_{\cc})^{-1}-\II_V)\theta^\vee\rangle \\ 
&=\frac{1}{n}\frac{p}{2p-1}\left\langle\theta^\vee,2p\zeta-\theta^\vee\right\rangle \\ 
&=\frac{1}{2n^2}\frac{p}{1-p}. 
\end{align*} 

Finally, we have $\sum_{\beta\in\Phi^+}\norm{\beta}=n(n-1)\sqrt{2}+2n$, so the formula for $\lim_{K\to\infty}\EE[\ell(v_p(\bb^K))]/\sqrt{K}$ in \cref{table:sigma_b} follows from \cref{cor:length} and the formula for $\sigma_\bb^2$.

\subsection{Type \texorpdfstring{$\widetilde B_r$}{widetilde{B}_r}} 
Let $\Phi$ be the root system of type $B_r$, where $r=n\geq 3$. That is, 
\[\Phi=\{\pm e_i\pm e_j:1\leq i<j\leq n\}\cup\{\pm e_i:i\in[n]\}.\] Then 
$V=\mathbb R^n$. For $1\leq i\leq n-1$, the $i$-th simple root is $\alpha_i=e_i-e_{i+1}$. Also, $\alpha_n=e_n$. We have \[\theta^\vee=e_1+e_2.\] 
The Coxeter graph of the affine Weyl group $W=\widetilde B_n$ is 
\[\begin{array}{l}\includegraphics[height=1.220cm]{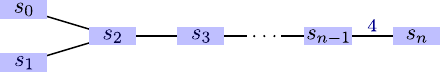}\end{array}.\]  

For $1\leq i\leq n-1$, the linear transformation $P_{s_i}$ swaps the vectors $e_i$ and $e_{i+1}$ and fixes $e_j$ for all ${j\in[n]\setminus\{i,i+1\}}$.  We have $P_{s_n}e_n=-e_n$ and $P_{s_n}e_j=e_j$ for all $j\in[n-1]$. In addition, we have $P_{\overline s_0}e_1=-e_2$, $P_{\overline s_0}e_2=-e_1$, and $P_{\overline s_0}e_j=e_j$ for all $j\in[3,n]$.

Let $\cc=s_n\cdots s_1s_0$. The Coxeter graph of $\widetilde B_n$ is a tree, so \cref{rem:flip} tells us that to prove the formula for $\sigma_\bb^2$ in \cref{table:sigma_b}, it suffices to prove it when $\bb=\cc$. 

For $1\leq k\leq n+1$, let $\aa_k=s_n\cdots s_{n-k+1}$ be the prefix of $\cc$ of length $k$. Let 
\[\zeta=\frac{1}{4(n-1)p(1-p)}\sum_{i=1}^n((2n-2i+1)(1-p)+p)e_i.\] It is straightforward to prove by induction on $k$ that 
\[R_{\aa_k}\zeta=\zeta-\frac{1}{2(n-1)(1-p)}e_{n-k+1}\] for all $1\leq k\leq n$. Therefore, 
\begin{align*}
R_{\cc}\zeta&=R_{s_0}R_{\aa_n}\zeta \\ 
&=R_{s_0}\left(\zeta-\frac{1}{2(n-1)(1-p)}e_1\right) \\ 
&=\zeta-\left(1+\frac{1}{2(n-1)(1-p)}\right)e_1-e_2.
\end{align*}
Since $R_{s_2},\ldots,R_{s_n}$ all fix $e_1$, we also have 
\begin{align*}
R_{\cc}e_1&=R_{s_0}R_{s_1}e_1 \\ 
&=R_{s_0}((1-p)e_1+pe_2) \\ 
&=(1-2p)e_1.
\end{align*} 
Consequently, 
\begin{align*}
(\II_V-R_{\cc})\left(\zeta-\frac{1}{4(n-1)p(1-p)}e_1\right)=e_1+e_2=\theta^\vee.
\end{align*}

The desired formula for $\sigma_\cc^2$ now follows from \cref{thm:covariance_simple}. Indeed, for $p\neq 1/2$, we have 
\begin{align*}
\sigma_{\cc}^2&=\frac{1}{n}\frac{p}{2p-1}\langle\theta^\vee,(2p(\II_V-R_{\cc})^{-1}-\II_V)\theta^\vee\rangle \\ 
&=\frac{1}{n}\frac{p}{2p-1}\left\langle\theta^\vee,2p\left(\zeta-\frac{1}{4(n-1)p(1-p)}e_1\right)-\theta^\vee\right\rangle \\ 
&=\frac{1}{2n(n-1)}\frac{p}{1-p}. 
\end{align*} 

Finally, we have $\sum_{\beta\in\Phi^+}\norm{\beta}=n(n-1)\sqrt{2}+n$, so the formula for $\lim_{K\to\infty}\EE[\ell(v_p(\bb^K))]/\sqrt{K}$ in \cref{table:sigma_b} follows from \cref{cor:length} and the formula for $\sigma_\bb^2$.

\subsection{Type \texorpdfstring{$\widetilde D_r$}{widetilde{D}_r}} 

Let $\Phi$ be the root system of type $D_r$, where $r=n\geq 4$. That is, 
\[\Phi=\{\pm e_i\pm e_j:1\leq i<j\leq n\}.\] Then $V=\mathbb R^n$. For $1\leq i\leq n-1$, the $i$-th simple root is $\alpha_i=e_i-e_{i+1}$. Also, $\alpha_n=e_{n-1}+e_n$. We have \[\theta^\vee=e_1+e_2.\] 

The Coxeter graph of the affine Weyl group $W=\widetilde D_n$ is \[\begin{array}{l}\includegraphics[height=1.220cm]{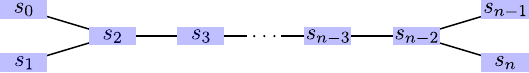}\end{array}.\] 

For $1\leq i\leq n-1$, the linear transformation $P_{s_i}$ swaps the vectors $e_i$ and $e_{i+1}$ and fixes $e_j$ for all $j\in[n]\setminus\{i,i+1\}$. We have $P_{s_n}e_{n-1}=-e_n$, $P_{s_n}e_n=-e_{n-1}$, and $P_{s_n}e_j=e_j$ for all $j\in[n-2]$. In addition, we have $P_{\overline s_0}e_1=-e_2$, $P_{\overline s_0}e_2=-e_1$, and $P_{\overline s_0}e_j=e_j$ for all $j\in[3,n]$. 

Let $\cc=s_n\cdots s_1s_0$. The Coxeter graph of $\widetilde D_n$ is a tree, so \cref{rem:flip} tells us that to prove the formula for $\sigma_\bb^2$ in \cref{table:sigma_b}, it suffices to prove it when $\bb=\cc$. 

For $1\leq k\leq n+1$, let $\aa_k=s_n\cdots s_{n-k+1}$ be the prefix of $\cc$ of length $k$. Let 
\[\zeta=\frac{1}{4(n-2)p(1-p)}\sum_{i=1}^{n-1}((2n-2i-1)(1-p)+p)e_{i}.\] 
One can check that \[R_{s_n}\zeta=\zeta-\frac{1}{4(n-2)(1-p)}(e_{n-1}+e_n).\]
It is then straightforward to prove by induction on $k$ that 
\[R_{\aa_k}\zeta=\zeta-\frac{1}{2(n-2)(1-p)}e_{n-k+1}\] for all $2\leq k\leq n$. Therefore, 
\begin{align*}
R_\cc\zeta&= R_{s_0}R_{\aa_n}\zeta \\ 
&=R_{s_0}\left(\zeta-\frac{1}{2(n-2)(1-p)}e_1\right) \\ 
&=\zeta-\left(1+\frac{1}{2(n-2)(1-p)}\right)e_1-e_2. 
\end{align*}
Since $R_{s_2},\ldots,R_{s_n}$ all fix $e_1$, we also have 
\begin{align*}
R_{\cc}e_1&=R_{s_0}R_{s_1}e_1 \\ 
&=R_{s_0}((1-p)e_1+pe_2) \\ 
&=(1-2p)e_1.
\end{align*} 
Consequently, 
\begin{align*}
(\II_V-R_{\cc})\left(\zeta-\frac{1}{4(n-2)p(1-p)}e_1\right)=e_1+e_2=\theta^\vee.
\end{align*} 

The desired formula for $\sigma_\cc^2$ now follows from \cref{thm:covariance_simple}. Indeed, for $p\neq 1/2$, we have 
\begin{align*}
\sigma_{\cc}^2&=\frac{1}{n}\frac{p}{2p-1}\langle\theta^\vee,(2p(\II_V-R_{\cc})^{-1}-\II_V)\theta^\vee\rangle \\ 
&=\frac{1}{n}\frac{p}{2p-1}\left\langle\theta^\vee,2p\left(\zeta-\frac{1}{4(n-2)p(1-p)}e_1\right)-\theta^\vee\right\rangle \\ 
&=\frac{1}{2n(n-2)}\frac{p}{1-p}. 
\end{align*} 

Finally, we have $\sum_{\beta\in\Phi^+}\norm{\beta}=n(n-1)\sqrt{2}$, so the formula for $\lim_{K\to\infty}\EE[\ell(v_p(\bb^K))]/\sqrt{K}$ in \cref{table:sigma_b} follows from \cref{cor:length} and the formula for $\sigma_\bb^2$. 

\subsection{Type \texorpdfstring{$\widetilde A_r$}{widetilde{A}_r}}\label{subsec:TypeA} 
Let $\Phi$ be the root system of type $A_r$, where $r=n-1\geq 1$. That is, 
\[\Phi=\{e_i-e_j:i,j\in[n],\, i\neq j\}.\] Then 
\[V=\{(\gamma_1,\ldots,\gamma_n)\in\mathbb R^n:\gamma_1+\cdots+\gamma_n=0\}.\] 
For $1\leq i\leq n-1$, the $i$-th simple root is $\alpha_i=e_i-e_{i+1}$. We have \[\theta^\vee=e_1-e_n.\] Consider the affine Weyl group $W=\widetilde A_{n-1}$. If $n\geq 3$, then the Coxeter graph of $W$ is 
\[\begin{array}{l}\includegraphics[height=1.220cm]{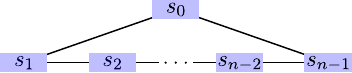}\end{array}.\] If $n=2$, then the Coxeter graph of $W$ is 
\[\begin{array}{l}\includegraphics[height=0.372cm]{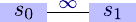}\end{array}.\]

For $0\leq i\leq n-1$, let $\PPP_{s_i}$ be the $n\times n$ permutation matrix corresponding to the transposition that swaps $i$ and $i+1$ (modulo $n$), and let $\RRR_{s_i}=p\PPP_{s_i}+(1-p)\II_{n}$, where $\II_n$ is the $n\times n$ identity matrix. We view $\PPP_{s_i}$ and $\RRR_{s_i}$ as linear transformations of $\mathbb R^n$ that agree with $P_{s_i}$ and $R_{s_i}$, respectively, on the subspace $V$. Given a finite word $\aa=s_{j_k}\cdots s_{j_1}$, let $\PPP_{\aa}=\PPP_{s_{j_1}}\cdots \PPP_{s_{j_k}}$ and $\RRR_{\aa}=\RRR_{s_{j_1}}\cdots \RRR_{s_{j_k}}$. 

For $1\leq d\leq n-1$, let $\cc^{(d)}$ be the word $s_ds_{d+1}\cdots s_{n-1}s_{d-1}\cdots s_1s_0$. If $n\geq 3$, then the Coxeter graph of $W$ is a cycle, so it follows readily from \cref{rem:flip} that every Coxeter word is cyclically commutation equivalent to a word of the form $\cc^{(d)}$. If $n=2$, then the only Coxeter words are $s_0s_1$ and $c^{(1)}=s_1s_0$, which are cyclically commutation equivalent. Therefore, according to \cref{lem:invariance}, proving the formula in \cref{table:sigma_b} for $\sigma_\bb^2$ reduces to proving that 
\begin{equation}\label{eq:type_A_sigma's}\sigma_{\cc^{(1)}}^2=\cdots=\sigma_{\cc^{(n-1)}}^2=\frac{2}{n(n-1)}\frac{p}{1-p}.
\end{equation} 

Fix $1\leq d\leq n-1$. For $1\leq k\leq n$, let $\aa_k$ be the prefix of $\cc^{(d)}$ of length $k$. For example, $\aa_1=s_d$, and $\aa_n=\cc^{(d)}=s_ds_{d+1}\cdots s_{n-1}s_{d-1}\cdots s_1s_0$. Let $\zeta=\sum_{i=1}^nie_i$. For $1\leq k\leq n$, let $\omega_k=\sum_{i=1}^ke_i$. 

\begin{lemma}\label{lem:computations_A}
For each $1\leq d\leq n-1$, we have  
\[(\II_V-R_{\cc^{(d)}})^{-1}\theta^\vee=\frac{1}{np}\left(\left(\frac{n+1}{2}-\frac{dp}{n(1-p)}\right)\omega_n+\frac{p}{1-p}\omega_d-\zeta\right).\]
\end{lemma}

\begin{proof}
We first claim that 
\begin{equation}\label{eq:cc1}
\RRR_{\aa_k}\zeta=\zeta+\sum_{i=0}^{k-1}p^{i+1}e_{d+i}-\frac{p-p^{k+1}}{1-p}e_{d+k}
\end{equation}
for all $1\leq k\leq n-d$. For example, if $n=8$, $d=5$, and $k=2$, then we have 
\begin{align*}\RRR_{\aa_2}\zeta&=\RRR_{s_6}\RRR_{s_5}\zeta \\ &=e_1+e_2+e_3+e_4+(5+p)e_5+(6+p^2)e_6+(7-p-p^2-p^3)e_7+e_8.
\end{align*}
The identity in \eqref{eq:cc1} is straightforward to prove by induction on $k$, so we omit the proof. Next, we claim that 
\begin{equation}\label{eq:cc2}
\RRR_{\aa_k}\zeta=\zeta+\sum_{i=1}^{n-d-1}p^{i+1}e_{d+i}-\sum_{j=1}^{k-n+d}p^{j+1}e_{d+1-j}-\frac{p-p^{n-d+1}}{1-p}e_{n}+\frac{p-p^{k-n+d+2}}{1-p}e_{n-k}
\end{equation}
for all $n-d\leq k\leq n-1$. For example, if $n=8$, $d=5$, and $k=6$, then 
\begin{align*}\RRR_{\aa_6}\zeta&=\RRR_{s_2}\RRR_{s_3}\RRR_{s_4}\RRR_{s_7}\RRR_{s_6}\RRR_{s_5}\zeta \\ 
&=e_1+(2+p+p^2+p^3+p^4)e_2+(3-p^4)e_3+(4-p^3)e_4+(5-p^2)e_5 \\&\,\,\,\,\,\,+(6+p^2)e_6+(7+p^3)e_7+(8-p-p^2-p^3)e_8. 
\end{align*}
As before, we omit the proof of \eqref{eq:cc2} because it is straightforward by induction on $k$ (using \eqref{eq:cc1} for the base case). Setting $k=n-1$ in \eqref{eq:cc2}, we find that 
\begin{align}
\nonumber \RRR_{\cc^{(d)}}\zeta&=\RRR_{s_0}\RRR_{\aa_{n-1}}\zeta \\ 
\nonumber &=\RRR_{s_0}\left(\zeta+\sum_{i=1}^{n-d-1}p^{i+1}e_{d+i}-\sum_{j=1}^{d-1}p^{j+1}e_{d+1-j}-\frac{p-p^{n-d+1}}{1-p}e_{n}+\frac{p-p^{d+1}}{1-p}e_{1}\right) \\ 
\label{eq:cc3} &=\RRR_{s_0}\zeta+\sum_{i=1}^{n-d-1}p^{i+1}e_{d+i}-\sum_{j=1}^{d-1}p^{j+1}e_{d+1-j}-\frac{p-p^{n-d+1}}{1-p}\RRR_{s_0}e_{n}+\frac{p-p^{d+1}}{1-p}\RRR_{s_0}e_{1}.  
\end{align} 
Now, we have 
\[\RRR_{s_0}\zeta=\zeta+(n-1)p(e_1-e_n),\quad \RRR_{s_0}e_n=pe_1+(1-p)e_n,\quad \RRR_{s_0}e_1=pe_n+(1-p)e_1\] 
so simplifying \eqref{eq:cc3} yields that 
\begin{align}
\nonumber (\II_n-\RRR_{\cc^{(d)}})\zeta=\,&-\sum_{i=1}^{n-d-1}p^{i+1}e_{d+i}+\sum_{j=1}^{d-1}p^{j+1}e_{d+1-j} \\ \label{eq:cc4}&-\left(np-\frac{p^2-p^{n-d+2}}{1-p}-p^{d+1}\right)e_1+\left(np-\frac{p^2-p^{d+2}}{1-p}-p^{n-d+1}\right)e_n.
\end{align}

It is straightforward to prove by induction on $k$ that
\begin{equation}\label{eq:cc5}
\RRR_{\aa_k}\omega_d=\omega_{d-1}+(1-p)\sum_{i=0}^{k-1}p^ie_{d+i}+p^ke_{d+k}
\end{equation}
for all $1\leq k\leq n-d$. We can then prove by induction on $k$ (using \eqref{eq:cc5} as a base case) that 
\begin{equation}\label{eq:cc6}
\RRR_{\aa_k}\omega_d=\omega_{d}+(1-p)\left(\sum_{i=1}^{n-d-1}p^ie_{d+i}-\sum_{j=1}^{k-n+d}p^je_{d+1-j}\right)+p^{n-d}e_{n}-p^{k-n+d+1}e_{n-k}  
\end{equation}
for all $n-d\leq k\leq n-1$. Setting $k=n-1$ in \eqref{eq:cc6}, we find that 
\begin{align*}
\RRR_{\cc^{(d)}}\omega_d&=\RRR_{s_0}\RRR_{\aa_{n-1}}\omega_d \\ 
&=\RRR_{s_0}\left(\omega_{d}+(1-p)\left(\sum_{i=1}^{n-d-1}p^ie_{d+i}-\sum_{j=1}^{d-1}p^je_{d+1-j}\right)+p^{n-d}e_{n}-p^{d}e_{1}  \right). 
\end{align*}
We have $\RRR_{s_0}\omega_d=\omega_d-p(e_1-e_n)$, so 
\begin{align}
\nonumber(\II_n-\RRR_{\cc^{(d)}})\omega_d&=p(e_1-e_n)-(1-p)\left(\sum_{i=1}^{n-d-1}p^ie_{d+i}-\sum_{j=1}^{d-1}p^je_{d+1-j}\right) \\ \nonumber&\,\,\,\,\,\,-p^{n-d}(pe_1+(1-p)e_n)+p^{d}(pe_n+(1-p)e_1) \\ 
\nonumber&=-(1-p)\left(\sum_{i=1}^{n-d-1}p^ie_{d+i}-\sum_{j=1}^{d-1}p^je_{d+1-j}\right) \\ \label{eq:cc7}&\,\,\,\,\,\,-(p^{n-d+1}+p^{d+1}-p^d-p)e_1+(p^{n-d+1}-p^{n-d}+p^{d+1}-p)e_n. 
\end{align} 
Also, the matrix $\RRR_{\cc^{(d)}}$ is row-stochastic, so 
\begin{equation}\label{eq:cc8}
(\II_n-\RRR_{\cc^{(d)}})\omega_n=0.
\end{equation} 

Let \[\upsilon=\frac{1}{np}\left(\left(\frac{n+1}{2}-\frac{dp}{n(1-p)}\right)\omega_n+\frac{p}{1-p}\omega_d-\zeta\right).\]
Invoking \eqref{eq:cc4}, \eqref{eq:cc7}, and \eqref{eq:cc8}, it is straightforward to check that $(\II_n-\RRR_{\cc^{(d)}})\upsilon=\theta^\vee$. It is also easy to check that $\upsilon\in V$. Since $\II_n-\RRR_{\cc^{(d)}}$ is a linear transformation of $\mathbb R^n$ whose restriction to $V$ is $\II_V-R_{\cc^{(d)}}$, this completes the proof. 
\end{proof} 

The proof of \eqref{eq:type_A_sigma's} now follows from \cref{thm:covariance_simple,lem:computations_A}. Indeed, for $p\neq 1/2$, we have 
\begin{align*}
\sigma_{\cc^{(d)}}^2&=\frac{1}{r}\frac{p}{2p-1}\langle\theta^\vee,(2p(\II_V-R_{\cc^{(d)}})^{-1}-\II_V)\theta^\vee\rangle \\ 
&=\frac{1}{n-1}\frac{p}{2p-1}\left\langle\theta^\vee,\frac{2}{n}\left(\left(\frac{n+1}{2}-\frac{dp}{n(1-p)}\right)\omega_n+\frac{p}{1-p}\omega_d-\zeta\right)-\theta^\vee\right\rangle \\ 
&=\frac{1}{n-1}\frac{p}{2p-1}\left(\frac{2}{n}\left(\frac{p}{1-p}\langle\theta^\vee,\omega_d\rangle-\langle\theta^\vee,\zeta\rangle\right)-\langle\theta^\vee,\theta^\vee\rangle\right) \\ 
&=\frac{1}{n-1}\frac{p}{2p-1}\left(\frac{2}{n}\left(\frac{p}{1-p}-(1-n)\right)-2\right) \\ 
&=\frac{2}{n(n-1)}\frac{p}{1-p}.  
\end{align*} 

Finally, we have $\sum_{\beta\in\Phi^+}\norm{\beta}=\frac{n(n-1)}{2}\sqrt{2}$, so the formula for $\lim_{K\to\infty}\EE[\ell(v_p(\bb^K))]/\sqrt{K}$ in \cref{table:sigma_b} follows from \cref{cor:length} and the formula for $\sigma_\bb^2$. 

\subsection{Type \texorpdfstring{$\widetilde E_6$}{widetilde{E}_6}}
Let $\Phi$ be the root system of type $E_6$. Then $r=6$. There are several different isomorphic ways of explicitly defining $\Phi$; we will find it convenient to let $n=8$ and define 
\[\Phi=\left\{(\beta_1,\ldots,\beta_8)\in\ZZ^8:\sum_{i=1}^8\beta_i^2=2,\,\sum_{i=1}^8\beta_i\in2\ZZ,\,\beta_6=\beta_7=-\beta_8\right\}.\]
Then \[V=\{(\gamma_1,\ldots,\gamma_8)\in\mathbb R^8:\gamma_6=\gamma_7=-\gamma_8\}.\]
We will choose the simple roots to be $\alpha_1,\ldots,\alpha_6$, where $\alpha_1=\frac{1}{2}(e_1-e_2-e_3-e_4-e_5-e_6-e_7+e_8)$, $\alpha_2=e_1+e_2$, and $\alpha_i=e_{i-1}-e_{i-2}$ for all $3\leq i\leq 6$. With these conventions, we have 
\[\theta^\vee=\frac{1}{2}(e_1+e_2+e_3+e_4+e_5-e_6-e_7+e_8).\] 
The Coxeter graph of the affine Weyl group $W=\widetilde E_6$ is 
\[\begin{array}{l}\includegraphics[height=2.120cm]{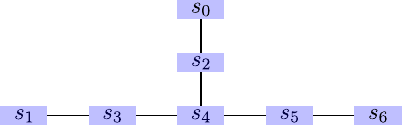}\end{array}.\] 

Let $\cc=s_6s_5s_4s_3s_2s_1s_0$. The Coxeter graph of $\widetilde E_6$ is a tree, so \cref{rem:flip} tells us that to prove the formula for $\sigma_\bb^2$ in \cref{table:sigma_b}, it suffices to prove it when $\bb=\cc$. Using \cref{thm:covariance_simple} and a computer, we have found that \[\sigma_\cc^2=\frac{1}{72}\frac{p}{1-p}.\] This yields the formula for $\sigma_\bb^2$ in \cref{table:sigma_b}, and the formula for $\lim_{K\to\infty}\EE[\ell(v_p(\bb^K))]/\sqrt{K}$ follows from \cref{cor:length} and the fact that $\sum_{\beta\in\Phi^+}\norm{\beta}=36\sqrt{2}$.

\subsection{Type \texorpdfstring{$\widetilde E_7$}{widetilde{E}_7}}
Let $\Phi$ be the root system of type $E_7$. Then $r=7$. There are several different isomorphic ways of explicitly defining $\Phi$; we will find it convenient to let $n=8$ and define 
\[\Phi=\left\{(\beta_1,\ldots,\beta_8)\in\ZZ^8:\sum_{i=1}^8\beta_i^2=2,\,\sum_{i=1}^8\beta_i\in2\ZZ,\,\beta_7=-\beta_8\right\}.\]
Then \[V=\{(\gamma_1,\ldots,\gamma_8)\in\mathbb R^8:\gamma_7=-\gamma_8\}.\]
We will choose the simple roots to be $\alpha_1,\ldots,\alpha_7$, where $\alpha_1=\frac{1}{2}(e_1-e_2-e_3-e_4-e_5-e_6-e_7+e_8)$, $\alpha_2=e_1+e_2$, and $\alpha_i=e_{i-1}-e_{i-2}$ for all $3\leq i\leq 7$. With these conventions, we have 
\[\theta^\vee=-e_7+e_8.\] 
The Coxeter graph of the affine Weyl group $W=\widetilde E_7$ is 
\[\begin{array}{l}\includegraphics[height=1.220cm]{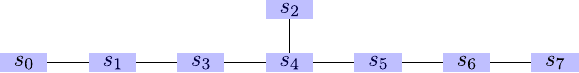}\end{array}.\] 

Let $\cc=s_7s_6s_5s_4s_3s_2s_1s_0$. The Coxeter graph of $\widetilde E_7$ is a tree, so \cref{rem:flip} tells us that to prove the formula for $\sigma_\bb^2$ in \cref{table:sigma_b}, it suffices to prove it when $\bb=\cc$. Using \cref{thm:covariance_simple} and a computer, we have found that \[\sigma_\cc^2=\frac{1}{168}\frac{p}{1-p}.\] This yields the formula for $\sigma_\bb^2$ in \cref{table:sigma_b}, and the formula for $\lim_{K\to\infty}\EE[\ell(v_p(\bb^K))]/\sqrt{K}$ follows from \cref{cor:length} and the fact that $\sum_{\beta\in\Phi^+}\norm{\beta}=63\sqrt{2}$. 

\subsection{Type \texorpdfstring{$\widetilde E_8$}{widetilde{E}_8}}
Let 
\[\Phi=\left\{(\beta_1,\ldots,\beta_8)\in\ZZ^8:\sum_{i=1}^8\beta_i^2=2,\,\sum_{i=1}^8\beta_i\in2\ZZ\right\}\] be the root system of type $E_8$. Then $r=n=8$, and $V=\mathbb R^8$. We will choose the simple roots to be $\alpha_1,\ldots,\alpha_8$, where $\alpha_1=\frac{1}{2}(e_1-e_2-e_3-e_4-e_5-e_6-e_7+e_8)$, $\alpha_2=e_1+e_2$, and $\alpha_i=e_{i-1}-e_{i-2}$ for all $3\leq i\leq 8$. With these conventions, we have 
\[\theta^\vee=e_7+e_8.\] 
The Coxeter graph of the affine Weyl group $W=\widetilde E_8$ is 
\[\begin{array}{l}\includegraphics[height=1.220cm]{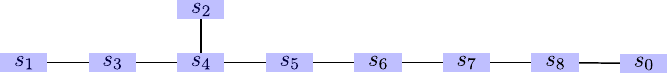}\end{array}.\] 

Let $\cc=s_8s_7s_6s_5s_4s_3s_2s_1s_0$. The Coxeter graph of $\widetilde E_8$ is a tree, so \cref{rem:flip} tells us that to prove the formula for $\sigma_\bb^2$ in \cref{table:sigma_b}, it suffices to prove it when $\bb=\cc$. Using \cref{thm:covariance_simple} and a computer, we have found that \[\sigma_\cc^2=\frac{1}{480}\frac{p}{1-p}.\] This yields the formula for $\sigma_\bb^2$ in \cref{table:sigma_b}, and the formula for $\lim_{K\to\infty}\EE[\ell(v_p(\bb^K))]/\sqrt{K}$ follows from \cref{cor:length} and the fact that $\sum_{\beta\in\Phi^+}\norm{\beta}=120\sqrt{2}$. 

\subsection{Type \texorpdfstring{$\widetilde F_4$}{widetilde{F}_4}} 
Let $r=n=4$, and let
\[\Phi=\{\pm e_i\pm e_j:1\leq i<j\leq 4\}\sqcup\{\pm e_i:1\leq i\leq 4\}\sqcup\{(\pm e_1\pm e_2\pm e_3\pm e_4)/2\}\] be the root system of type $F_4$. Then $V=\mathbb R^4$. We choose the simple roots to be 
\[\alpha_1=e_2-e_3,\quad \alpha_2=e_3-e_4,\quad \alpha_3=e_4,\quad\alpha_4=\frac{1}{2}(e_1-e_2-e_3-e_4).\] With these conventions, we have 
\[\theta^\vee=e_1+e_2.\] 
The Coxeter graph of the affine Weyl group $W=\widetilde F_4$ is 
\[\begin{array}{l}\includegraphics[height=0.448cm]{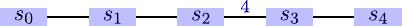}\end{array}.\] 

Let $\cc=s_4s_3s_2s_1s_0$. The Coxeter graph of $\widetilde F_4$ is a tree, so \cref{rem:flip} tells us that to prove the formula for $\sigma_\bb^2$ in \cref{table:sigma_b}, it suffices to prove it when $\bb=\cc$. Using \cref{thm:covariance_simple} and a computer, we have found that \[\sigma_\cc^2=\frac{1}{48}\frac{p}{1-p}.\] This yields the formula for $\sigma_\bb^2$ in \cref{table:sigma_b}, and the formula for $\lim_{K\to\infty}\EE[\ell(v_p(\bb^K))]/\sqrt{K}$ follows from \cref{cor:length} and the fact that $\sum_{\beta\in\Phi^+}\norm{\beta}=12\sqrt{2}+12$. 

\subsection{Type \texorpdfstring{$\widetilde G_2$}{widetilde{G}_2}} 
Let $\Phi$ be the root system of type $G_2$. Then $r=2$. We find it convenient to let $n=3$ and define
\[\Phi=\{e_i-e_j:i,j\in[3],\,i\neq j\}\sqcup\{\pm (e_i+e_j-2e_k):\{i,j,k\}=\{1,2,3\}\}.\] Then \[V=\{(\gamma_1,\gamma_2,\gamma_3)\in\mathbb R^3:\gamma_1+\gamma_2+\gamma_3=0\}.\] We choose the simple roots to be 
\[\alpha_1=e_2-e_3,\quad \alpha_2=e_1-2e_2+e_3.\] With these conventions, we have 
\[\theta^\vee=\frac{1}{3}(2e_1-e_2-e_3).\] 
The Coxeter graph of the affine Weyl group $W=\widetilde G_2$ is 
\[\begin{array}{l}\includegraphics[height=0.448cm]{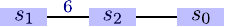}\end{array}.\] 

Let $\cc=s_2s_1s_0$. The Coxeter graph of $\widetilde G_2$ is a tree, so \cref{rem:flip} tells us that to prove the formula for $\sigma_\bb^2$ in \cref{table:sigma_b}, it suffices to prove it when $\bb=\cc$. It is straightforward to show that 
\[(\II_V-R_\cc)^{-1}\theta^\vee=\frac{1}{24p(1-p)}\left((7-6p)e_1-2e_2+(-5+6p)e_3\right).\] Applying \cref{thm:covariance_simple}, we find that \[\sigma_\cc^2=\frac{1}{24}\frac{p}{1-p}.\] This yields the formula for $\sigma_\bb^2$ in \cref{table:sigma_b}, and the formula for $\lim_{K\to\infty}\EE[\ell(v_p(\bb^K))]/\sqrt{K}$ follows from \cref{cor:length} and the fact that $\sum_{\beta\in\Phi^+}\norm{\beta}=3\sqrt{2}+3\sqrt{6}$.  

\section{Concluding Remarks}\label{sec:conclusion}  

In \cref{thm:table}, we gave a table of explicit formulas that tell us the value of $\sigma_\bb^2$ when $\bb$ is a Coxeter word of an arbitrary affine Weyl group $W$. It would be interesting to find a more conceptual explanation for why these formulas are so simple, perhaps expressing the various formulas in \cref{table:sigma_b} as a single type-independent formula. 

We also believe that there should be natural choices of $\bb$ besides Coxeter words such that $\sigma_\bb^2$ has a simple formula. For example, suppose $W=\widetilde A_{n-1}$, and let $\bb_{\theta^\vee}=s_1s_2\cdots s_{n-1}s_ns_{n-1}\cdots s_2s_1s_0$. The infinite word $\bb_{\theta^\vee}^\infty$ is of the form $\mathsf{w}_{\zz_0}(\theta^\vee)$ (as defined in \cref{subsec:billiards}) for some point $\zz_0$ in the interior of the fundamental alcove $\AA$. By performing a computation similar to the one for a Coxeter word in \cref{subsec:TypeA}, one can show that $(\II_V-R_{\bb_{\theta^\vee}})^{-1}\theta^\vee$ is equal to 
\[\frac{1}{n(n+2)p}\left(\left(\textstyle{\frac{n}{1-p}}+\textstyle{\binom{n-1}{2}}\right)e_1+\left(\textstyle{\frac{n}{1-p}}+\left(\textstyle{\binom{n}{2}}-1\right)\right)e_n+\sum_{i=2}^{n-1}\left(\textstyle{\binom{n+1}{2}}+1-in\right)e_i\right).\] One can then apply \cref{thm:covariance_simple} to find that 
\[\sigma_{\bb_{\theta^\vee}}^2=\frac{4}{(n-1)(n+2)}\frac{p}{1-p}.\] 

It is natural to ask for conditions on the word $\bb$ guaranteeing that $\sigma_\bb^2\frac{1-p}{p}$ is independent of $p$. 

We have only dealt with random billiard walks in which the initial direction of the beam of light is a positive scalar multiple of a coroot vector (i.e., a \emph{rational direction}). It would be very interesting to understand the behavior of random billiard walks when the initial directions can be arbitrary. This could likely require tools from ergodic theory. 

It would be interesting to have a better understanding of the mixing properties of the Markov chain $\mathscr{M}$ introduced in \cref{sec:normality}, especially when $W=\widetilde A_{n-1}$ and $\bb=s_{n-1}\cdots s_1s_0$. In this case, the Markov chain has the following natural combinatorial description. The state space is the symmetric group $\mathfrak S_n$ of permutations of $[n]$. For $i\in\ZZ/n\ZZ$, let $\mathbf{h}_i$ be the random operator on $\mathfrak S_n$ that swaps the numbers $i$ and $i+1$ with probability $p$ and does nothing with probability $1-p$ (in particular, $\mathbf{h}_0$ either swaps $n$ and $1$ or does nothing). One transition in $\mathscr{M}$ consists of applying the random operators $\mathbf{h}_0,\mathbf{h}_1,\ldots,\mathbf{h}_{n-1}$ in this order. This Markov chain is similar to, but not the same as, the cyclic adjacent transposition shuffle \cite{NN}.  

\section*{Acknowledgments}
Colin Defant was supported by the National Science Foundation under Award No.\ 2201907 and by a Benjamin Peirce Fellowship at Harvard University. 

Pakawut Jiradilok was supported by Elchanan Mossel's Vannevar Bush Faculty Fellowship ONR-N00014-20-1-2826 and by Elchanan Mossel's Simons Investigator award (622132). Part of this work was done when Jiradilok was at Mahidol University International College (MUIC), Nakhon Pathom, Thailand. Jiradilok would like to express his gratitude towards MUIC, especially its Science Division, for its hospitality.

Elchanan Mossel is partially supported by 
Vannevar Bush Faculty Fellowship ONR-N00014-20-1-2826 and by Simons Investigator award (622132).


\begin{thebibliography}{9999999999}

\bibitem[ADS24]{DefantRefractions}
A. Adams, C. Defant, and J. Striker. Toric promotion with reflections and refractions. \href{https://arxiv.org/abs/2404.03649}{arXiv:2404.03649}. 

\bibitem[Ale02]{Alexopoulos}
G. K. Alexopoulos. Random walks on discrete groups of polynomial volume growth. \emph{Ann. Probab.}, {\bf 30} (2002), 723--801. 

\bibitem[BDHKL24]{BDHKL}
G. Barkley, C. Defant, E. Hodges, N. Kravitz, and M. Lee. Bender--Knuth billiards in Coxeter groups. To appear in \emph{Forum Math. Sigma}, (2025). 

\bibitem[Bil95]{Bil95} P. Billingsley, Probability and measure, 3rd ed. Chichester:\ John Wiley \& Sons Ltd. (1995).

\bibitem[Def24A]{DefantStoned}
C. Defant. Random combinatorial billiards and stoned exclusion processes. \href{https://arxiv.org/abs/2406.07858}{arXiv:2406.07858}.

\bibitem[Def24B]{DefantPipes}
C. Defant. Random subwords and pipe dreams. \href{https://arxiv.org/abs/2408.05182}{arXiv:2408.05182}.  

\bibitem[DJ23]{DefantJiradilok}
C. Defant and P. Jiradilok. Triangular-grid billiards and plabic graphs. \emph{Comb. Theory}, {\bf 3} (2023). 

\bibitem[DMR16]{DMR}
M. Develin, M. Macauley, and V. Reiner. Toric partial orders. \emph{Trans. Amer. Math. Soc.}, {\bf 368} (2016), 2263--2287. 

\bibitem[MPPY24]{MPPY}
A. H. Morales, G. Panova, L. Petrov, and D. Yeliussizov. Grothendieck shenanigans: permutons from pipe dreams via integrable probability. \href{https://arxiv.org/abs/2407.21653}{arXiv:2407.21653}. 

\bibitem[NN]{NN}
D. Nam and E. Nestoridi. Cutoff for the cyclic adjacent transposition shuffle. \emph{Ann. Appl. Probab.}, {\bf 29} (2019), 3861--3892. 

\bibitem[Spe09]{Speyer}
D. Speyer. Powers of Coxeter elements in infinite groups are reduced. \emph{Proc. Amer. Math. Soc.}, {\bf 137} (2009), 1295--1302. 

\bibitem[Spi76]{Spi76}
F. Spitzer. Principles of random walk, 2nd ed. New York, NY: Springer, 1976. 

\bibitem[Woe00]{Woess}
W. Woess. Random Walks on infinite graphs and groups. Cambridge University Press, 2000.

\bibitem[Zhu23]{Zhu}
H. Zhu. The maximum number of cycles in a triangular-grid billiards system with a given perimeter. \href{https://arxiv.org/abs/2309.00100}{arXiv:2309.00100}. 

\end{thebibliography}
\end{document}